\documentclass[12pt]{article}

\newenvironment{keywords}{\paragraph{Keywords:}}{}
\newenvironment{acknowledgements}{\paragraph{Acknowledgements:}}{}

\usepackage[latin1]{inputenc}
\usepackage{amsfonts}
\usepackage{amssymb}
\usepackage{amsmath}
\usepackage{color}
\usepackage[mathscr]{euscript}
\usepackage{wasysym}
\usepackage{stmaryrd}
\usepackage{graphicx}
\usepackage{latexsym}

\usepackage[colors]{optsys}

\def\mathscr{\EuScript} 

\renewcommand{\va}[1]{\vaC{#1}}
\renewcommand{\opt}{^{\sharp}}        

\newcommand{\tini}{t_{i}}
\newcommand{\tfin}{t_{f}}
\newcommand{\tzero}{t_{0}}

\newcommand{\Truncat}{\mathfrak{T}}

\newcommand\Pbvalue[2]{{\cal V}_{\cal #1}\nc{\va{#2}}}

\newcommand\Pbvalueter[2]{{\cal V}_{{\cal #1}_i}\nc{#2}}

\newcommand\sumi[1]{\sum_{i=0}^{N'} #1 \findi{w'_i}(\va{w}')}
\newcommand\pisumi[1]{\sum_{i=0}^{N'} \pi'_i #1}
\newcommand\sumj[1]{\sum_{j=0}^{N''} #1 \findi{w''_j}(\va{w}'')}
\newcommand\pisumj[1]{\sum_{j=0}^{N''} \pi''_j #1}
\newcommand\sumij[1]{\sum_{i=0}^{N'}\sum_{j=0}^{N''} #1 \findi{w'_i}(\va{w}')\findi{w''_j}(\va{w}'')}

\usepackage{fullpage}
\usepackage{url}
\newcommand{\figscale}{0.40}

\title{Time Consistency \\
       for Multistage Stochastic Optimization Problems \\
       under Constraints in Expectation}

\author{%
  Pierre Carpentier\footnote{UMA, ENSTA Paris, Institut Polytechnique de Paris, Palaiseau, France},
  \and
  Jean-Philippe Chancelier\footnote{CERMICS, \'Ecole des Ponts, Marne-la-Vall\'ee, France},
  \and
  Michel De Lara\footnotemark[2]
}

\date{\today}

\begin{document}
\maketitle

\begin{abstract}
  We consider sequences ---  indexed by time (discrete stages) ---
  of families of multistage stochastic optimization problems.
  At each time, the optimization problems in a family are
  parameterized by some quantities (initial states, constraint
  levels\ldots). In this framework, we introduce an adapted
  notion of time consistent optimal solutions, that is,
  solutions that remain optimal after truncation of the past
  and that are optimal for any values
  of the parameters. We link this time consistency notion
  with the concept of state variable in Markov Decision
  Processes for a class of multistage stochastic optimization
  problems incorporating state constraints at the final
  time, either formulated in expectation or in probability.
  For such problems, when the primitive noise random process is stagewise
  independent and
  takes a finite number of values, we show that time consistent
  solutions can be obtained by considering a finite dimensional
  state variable. We illustrate our results on a simple dam
  management problem.
\end{abstract}



\setcounter{tocdepth}{2}


\section{Introduction and motivation}
\label{sect:introduction}

The notion of time consistency has been introduced in the field
of Economics \cite{Hammond:1989}, and developed in the context of
risk measures \cite{Artzner_AOR_2007,Riedel_SPA_2004,Detlefsen_FS_2005,
  Cheridito_EJP_2006}. It has been studied in stochastic optimization,
both from the stochastic programming \cite{Shapiro_ORL_2009,Pflug2016}
and from  the Markov Decision Process  \cite{Ruszczynski_OO_2009}
points of view.
Loosely speaking, time
consistency means that strategies obtained by solving the problem at the
very first stage do not have to be questioned later on. This definition
has been used in~\cite{Carpentier-Chancelier-Cohen-DeLara-Girardeau:2012}
to establish links between the concept of state variable and the notion
of time consistency.
The aim in~\cite{Carpentier-Chancelier-Cohen-DeLara-Girardeau:2012}
was to highlight the role of information in time consistency. For example,
considering a standard multistage stochastic optimization problem
solvable by dynamic programming, it was shown that adding a probabilistic
constraint involving the state at the final instant of the time span
invalidates the time inconsistency property, in the sense that optimal
strategies based on the usual state variable have to be reconsidered
at each time stage. It was also shown that it was possible to devise an
appropriate state variable, namely the probabilistic distribution
of the state variable rather than the state variable itself, to
formulate an equivalent problem enjoying the time consistency
property. But this state is an infinite dimensional one, so that
dynamic programming is usually not implementable.
The aim of this article is to give deeper insights into the
results established in~\cite{Carpentier-Chancelier-Cohen-DeLara-Girardeau:2012}
and to show that it is possible to regain time consistency on such problems
by using an extended finite dimensional state variable.

The paper is organized as follows.
In Sect.~\ref{sect:uctc-soc},
we introduce the notion of universal solution for a family
of optimization problems, and we define the notion of time
consistency for a sequence of families of optimization
problems. Then, we revisit the setting of a discrete time
multistage stochastic optimization problem in the standard formulation,
and we show that our definition of time consistency applies
in this case.
In Sect.~\ref{sec:constraint}, we add an expectation
constraint on the final state to the standard multistage stochastic
optimization problem, and we define families of optimization
problems parameterized by both the initial state and the level
of constraint. We prove that the feedback strategies --- obtained by
dynamic programming on an extended problem formulation with
additional state and control variables --- are time consistent.
In Sect.~\ref{sect:numer},
we present a toy problem for managing a dam subject to a final
constraint in probability, and we give results obtained using
the extended formulation. Finally, we draw some conclusions
in Sect.~\ref{sect:conclusion}.

\section{Time consistency and multistage stochastic optimization}
\label{sect:uctc-soc}

In \S\ref{ssect:uctc}, we introduce the notion of universal
solution of a family of optimization problems and the notion
of time consistency of a sequence of controls for
an optimization data set. In \S\ref{ssec:classic}, we show how
these two notions apply in multistage stochastic optimization.

\subsection{Universal solutions and time consistency}
\label{ssect:uctc}

We start with general considerations on universal solutions and time
consistency,
before moving to more formal statements.

In optimization, the most natural notion of \emph{universal solution}
is the following. Let~$\espace{A}$ be a set of parameters
and~$\CONTROL$ be a set of (decision) variables.
Let \( \sequence{\Criterion^{a}}{a\in\espacea{A}} \) be a family
of functions\footnote{%
  {Adopting usage in mathematics, we follow Serge Lang and use ``function'' only to
    refer to mappings in which the codomain is numerical --- that is, a set of numbers (i.e. a subset
    of~$\RR$ or $\CC$, or their possible extensions with $\pm \infty$)
    --- and reserve the term ``mapping'' for more general codomains.}
  \label{ft:mapping_vs_function}}  \( \Criterion^{a} : \CONTROL \to \RR\cup\na{+\infty} \)
indexed by the parameter~$a$. An element \( \control\opt\in\CONTROL \)
is a \emph{universal solution} for the family of functions
\( \sequence{\Criterion^{a}}{a\in\espacea{A}} \) when
\( \control\opt\in\bigcap_{a\in\espacea{A}}\argmin_{\control\in\CONTROL}
\Criterion^{a}\np{\control} \).
In a different way, the most natural
notion of \emph{time consistency} in (mutistage) optimization is the following.
Let \( \tini < \tfin \) be two integers,
$\ic{\tini,\tfin}=\na{\tini,\tini{+}1,\dots,\tfin{-}1,\tfin}$ be the
corresponding finite time span, and $\sequence{\CONTROL_{t}}{t\in\ic{\tini,\tfin}}$
be a sequence of control sets. 
We introduce the \emph{truncation mapping} (projection) at time~$t \in\ic{\tini,\tfin}$, that is,
\( \Truncat_t : \CONTROL_{\tini}\times\dots\times\CONTROL_{\tfin}
   \to \CONTROL_{t}\times\dots\times\CONTROL_{\tfin} \) and
\( \Truncat_t (\control_{\tini},\ldots,\control_{\tfin}) =
(\control_{t},\ldots,\control_{\tfin}) \).
Then, considering
a sequence of functions $\sequence{\Criterion_{t}}{t\in\ic{\tini,\tfin}}$, with
\( \Criterion_{t} : \CONTROL_{t}\times\dots\times\CONTROL_{\tfin}
\to \RR\cup\na{+\infty} \), we say that \emph{time consistency} holds when,
for all~$t \in \ic{\tini,\tfin}$,
\begin{equation*}
  \Truncat_t \bp{ \argmin_{\np{\control_{\tini},\ldots,\control_{\tfin}} \in
    \CONTROL_{\tini}\times\dots\times\CONTROL_{\tfin} }
  \Criterion_{\tini}\np{\control_{\tini},\ldots,\control_{\tfin}} }
\subset \argmin_{\np{\control_{t},\ldots,\control_{\tfin}} \in
    \CONTROL_{t}\times\dots\times\CONTROL_{\tfin} }
  \Criterion_{t}\np{\control_{t},\ldots,\control_{\tfin}}
  \eqfinp
\end{equation*}
We now extend and mix these two notions in the case where the (cost)
functions depend on both parameters and time.

\begin{definition}
\label{def:optimization_data_set}
We call \emph{optimization data set} a family
$\mathscr{D} =\big( \mathcal{T}$, $\na{\espacea{A}_{t}}_{t\in\mathcal{T}}$,
$\na{\CONTROL_{t}}_{t\in\mathcal{T}}$,
$\na{J_{t}}_{t\in\mathcal{T}}\big)$, where
$\mathcal{T}=\ic{\tini,\tfin}$ (with  \( \tini < \tfin \) two integers),
a sequence~$\na{\espacea{A}_{t}}_{t\in\mathcal{T}}$ of parameter sets,
a sequence~$\na{\CONTROL_{t}}_{t\in\mathcal{T}}$ of control sets,
a sequence~$\na{J_{t}}_{t\in\mathcal{T}}$ of cost functions,
with \(   J_{t} :   \espacea{A}_{t}\times\espacea{U}_{t}\times\dots\times\CONTROL_{\tfin} \to  \mathbb{R}\cup\{+\infty\} \).

For any $t \in \mathcal{T}$, we call \emph{truncated optimization
data set at time~$t$} the optimization data set
\( \mathscr{D}_{t} =\bp{\ic{t,\tfin},\na{\espacea{A}_{s}}_{s\in\ic{t,\tfin}},
  \na{\CONTROL_{s}}_{s\in\ic{t,\tfin}},\na{J_{s}}_{s\in\ic{t,\tfin}}} \).
\end{definition}



The notion of universal solution for an optimization data set
at time $t \in \mathcal{T}$ is the following.

\begin{definition}
  \label{def:universalsolution}
  Let~${\mathscr{D}}
      $
      be an optimization data set and let $t \in \mathcal{T}$ be given.
      We say that $(u_{t}\opt,\dots,u_{\tfin}\opt)\in\CONTROL_{t}\times\dots\times\CONTROL_{\tfin}$
  is a \emph{universal solution for the data set~${\mathscr{D}}$ at time~$t$}
  if it satisfies
  \begin{equation}
    (u_{t}\opt,\dots,u_{\tfin}\opt) \in
       \bigcap_{a_{t}\in\espacea{A}_{t}} \argmin  \ba{\mathscr{P}_{t}^{\mathscr{D}}(a_{t})}
    \eqfinv
  \end{equation}
  where, for any $a_t\in\espacea{A}_{t}$, the optimization problem~$\mathscr{P}_{t}^{\mathscr{D}}(a_{t})$
  is defined by
  \begin{equation}
    \label{pb:general-Jt}
    \min_{(u_{t},\dots,u_{\tfin})\in
          \CONTROL_{t}\times\dots\times\CONTROL_{\tfin}} \;
    J_{t} (a_{t},u_{t},\dots,u_{\tfin}) \eqfinp
  \end{equation}
\end{definition}


The property of \emph{time consistency} of a
sequence~$(u_{\tini}\opt,\dots,u_{\tfin}\opt)$ of controls
for an optimization data set~$\mathscr{D}$
is defined as follows.

\begin{definition}
  \label{def:timeconsistency}
  Let~$\mathscr{D}$ be an optimization data set and let
  $(u_{\tini}\opt,\dots,u_{\tfin}\opt)$ be a sequence of controls
  in $\CONTROL_{\tini}\times\dots\times\CONTROL_{\tfin}$.
  We say that the sequence~$(u_{\tini}\opt,\dots,u_{\tfin}\opt)$ of controls
  is \emph{time consistent for the optimization data set~$\mathscr{D}$}
  if, for any~$t\in\mathcal{T}$, the truncated subsequence
  $(u_{t}\opt,\dots,u_{\tfin}\opt)=\Truncat_{t}(u_{\tini}\opt,\dots,u_{\tfin}\opt)$
  of controls
  is a universal solution for the optimization data set~${\mathscr{D}}$ at time~$t$.
\end{definition}
Otherwise stated, given a universal solution
$(u_{\tini}\opt,\dots,u_{\tfin}\opt)$ for the data set~${\mathscr{D}}$
at initial time~$\tini$, time consistency means that the subsequence
$(u_{t}\opt,\dots,u_{\tfin}\opt)$ is a universal solution
for the data set~${\mathscr{D}}$ at time~$t$, for any time~$t>\tini$.

\subsection{Multistage stochastic optimization in the classical case}
\label{ssec:classic}

We study the standard case of a controlled dynamical system
influenced by exogenous disturbances. The decision
maker has to find strategies to drive the system so as to
minimize some objective function over a certain time span.


Let be given a probability space~$(\omeg, \trib, \prbt)$.
All random variables and all random processes are defined
on~$(\omeg, \trib, \prbt)$, and we denote them using bold letters.
We denote by $\sigma\np{\va{z}} \subset \trib$ the~$\sigma$-field
generated by a random variable~$\va{z}$.

We consider a positive integer $T>0$ and the finite time
span~$\ic{0,T}=\{0$, $1$, $\dots$,$T{-}1$, $T\}$. We denote by
$\va{w} = \na{\va{w}_t}_{t=1, \dots, T}$ the primitive (or exogenous)
\emph{noise random process}, where each random variable~$\va{w}_t$
takes values in a measurable space~$\espacea{W}_t$.
We denote by $\va{u}=\na{\va{u}_t}_{t=0,\dots,T{-}1}$ the
\emph{control random process}, where each random variable~$\va{u}_t$
takes values in a measurable space~$\espacea{U}_t$, and by
$\va{x}=\na{\va{x}_t}_{t=0,\dots,T}$ the
\emph{state random process}, where each random variable~$\va{x}_t$
takes values in a measurable space~$\espacea{X}_t$.
We consider a sequence \( \na{f_t}_{t=0,\dots,T{-}1} \) of
measurable mappings~$f_t : \espacea{X}_t\times\espacea{U}_t\times\espacea{W}_{t+1}
\rightarrow \espacea{X}_{t+1}$ (dynamics),
a sequence \( \na{L_t}_{t=0,\dots,T{-}1} \) of measurable functions
  $L_t : \espacea{X}_t\times\espacea{U}_t\times\espacea{W}_{t+1}
  \rightarrow \RR_+ \cup\na{+\infty}$ (instantaneous cost),
  and a measurable function
  $K : \espacea{X}_{T+1} \rightarrow \RR_+ \cup\na{+\infty}$ (final cost).

The optimization problem we consider below consists in minimizing
the expectation of a sum of costs depending on the state, the control
and the noise variables over the finite time span~$\ic{0,T}$.
The state variable evolves with respect to the dynamics~$f_t$ that
depends on the current state, noise and control values.
The problem starting at time~$t=0$ is
\begin{subequations}
  \label{pb:classic}
  \begin{align}
    \min_{\va{u},\va{x}} \quad
    & \bgesp{\sum_{t=0}^{T{-}1}
      L_t \np{\va{x}_t,\va{u}_t,\va{w}_{t+1}} + K \np{\va{x}_T}}
      \eqfinv \label{pb:classic-crit} \\
    \text{s.t.} \quad
    & \va{x}_{0} = x_{0} \eqfinv \label{pb:classic-ini} \\
    & \va{x}_{t+1} = f_t \np{\va{x}_t,\va{u}_t,\va{w}_{t+1}} \eqsepv
      \qquad \forall t=0, \dots, T{-}1 \eqfinv  \label{pb:classic-dyn}\\
    & \sigma\np{\va{u}_t} \subset
      \sigma\np{
      \va{w}_{1},\dots,\va{w}_t} \eqsepv
      \qquad\; \forall t = 0, \dots, T{-}1
      \eqfinp
      \label{pb:classic-mes}
  \end{align}
\end{subequations}
  By convention, for $t=0$, the $\sigma$-field
  $\sigma\np{\va{w}_{1},\dots,\va{w}_t}$ is the trivial $\sigma$-field
  $\na{\emptyset, \omeg}$.
Under the measurability assumptions made,
Problem~\eqref{pb:classic} is well-defined
as all functions take extended nonnegative values.

We make the following assumption.
\begin{assumption}[Markovian setting]
  \label{hyp:Markov}
  The noise random variables~$\va{w}_{1}, \dots, \va{w}_T$
  are independent.
\end{assumption}
Using Assumption~\ref{hyp:Markov}, it is well
known~\cite{Bertsekas_AS_2005} that
there is no loss of optimality in looking for the optimal
control~$\va{u}_t$ at time~$t$ of Problem~\eqref{pb:classic}
as a feedback strategy depending
on the state variable~$\va{x}_t$, that is, as a measurable
mapping~$\phi_{t}: \espacea{X}_t \rightarrow \espacea{U}_t$ (state feedback).

Let us embed Problem~\eqref{pb:classic} in the framework
developed in \S\ref{ssect:uctc}. For that purpose, we build an optimization data set
\begin{equation}
  \label{eq:dataset-classic}
  \mathscr{S} =\bp{\mathcal{T},
    \na{\espacea{X}_{t}}_{t\in\mathcal{T}},
    \na{\espacef{U}_{t}}_{t\in\mathcal{T}},
    \na{J_{t}}_{t\in\mathcal{T}}}\eqfinp
\end{equation}
The discrete time
span~$\mathcal{T}$ is $\ic{0,T{-}1}$,
the sequence of parameter sets is
$\na{\espacea{X}_{t}}_{t\in\mathcal{T}}$, and the sequence
of control spaces is~$\na{\espacef{U}_{t}}_{t\in\mathcal{T}}$,
$\espacef{U}_{t}$ being
the space of measurable mappings
$\phi_{t}: \espacea{X}_t \rightarrow \espacea{U}_t$
(state feedbacks).
The sequence~$\na{J_{t}}_{t\in\mathcal{T}}$ of cost functions
\begin{align*}
  J_{t} \; : \;
  & \espacea{X}_{t}\times\espacef{U}_{t}\times\dots\times\espacef{U}_{T{-}1}
    \; \longrightarrow \; \mathbb{R}\cup\{+\infty\} \\
  & \quad (x_{t},\phi_{t},\dots,\phi_{T{-}1}) \quad \longmapsto\;
    J_{t} (x_{t},\phi_{t},\dots,\phi_{T{-}1}) \eqfinv
\end{align*}
is defined by
\begin{align*}
  J_{t} (x_{t},\phi_{t},\dots,\phi_{T{-}1}) = \;
  & \bgesp{\sum_{\tau=t}^{T{-}1}
    L_\tau \bp{\va{x}_\tau,\phi_{\tau}\np{\va{x}_\tau},\va{w}_{\tau+1}} +
    K \np{\va{x}_T}} \eqfinv \\
  \text{with} \quad
  & \va{x}_{t} = x_{t} \eqfinv \\
  & \va{x}_{\tau+1} =
    f_\tau \bp{\va{x}_\tau,\phi_{\tau}\np{\va{x}_\tau},\va{w}_{\tau+1}}
   \eqsepv \forall \tau \in \ic{t,T{-}1} \eqfinp
\end{align*}

Thanks to dynamic programming, we obtain the following result.
The sequence $\np{\phi_{0}\opt, \dots, \phi_{T{-}1}\opt}$ of optimal strategies
of Problem~\eqref{pb:classic}, obtained by solving the dynamic
programming equation backward in time
\begin{subequations}
  \label{eq:classicBellman}
  \begin{align}
    V_{T}(x)
    & = K(x) \eqfinv \\
    V_{t}(x)
    & = \min_{u\in\espacea{U}_{t}} \Besp{L_{t}(x,u,\va{w}_{t+1}) +
      V_{t+1}\bp{f_{t}(x,u,\va{w}_{t+1})}} \eqfinv
    \intertext{with}
    \phi_{t}\opt(x)
    & \in \argmin_{u\in\espacea{U}_{t}} \Besp{L_{t}(x,u,\va{w}_{t+1}) +
      V_{t+1}\bp{f_{t}(x,u,\va{w}_{t+1})}} \eqfinv
  \end{align}
\end{subequations}
is time consistent, in the sense of Definition~\ref{def:timeconsistency}, for the
optimization data set~${\cal S}$ defined in Equation~\eqref{eq:dataset-classic}.
Indeed, letting time~$t \in \mathcal{T}$ be given, we build from the data set~${\cal S}$
the family
$\mathscr{P}_{t}^{\mathscr{S}}=
\ba{\mathscr{P}_{t}^{\mathscr{S}}(x_{t})}_{x_{t}\in\espacea{X}_{t}}$
of optimization problems as in Equation~\eqref{pb:general-Jt}, with
Problem~$\mathscr{P}_{t}^{\mathscr{S}}(x_{t})$ being
  \begin{equation}
    \label{pb:classic-Jt}
    \min_{(\phi_{t},\dots,\phi_{T{-}1})\in
      \espacef{U}_{t}\times\dots\times\espacef{U}_{T{-}1}} \;
    J_{t} (x_{t},\phi_{t},\dots,\phi_{T{-}1}) \eqfinp
  \end{equation}
  It is clear that Problem~$\mathscr{P}_{0}^{\mathscr{S}}(x_{0})$
  coincides with Problem~\eqref{pb:classic}.
  From the Bellman theory, we know that the sequence
  $\np{\phi_{0}\opt, \dots, \phi_{T{-}1}\opt}$ of strategies obtained by solving
  the dynamic programming equation~\eqref{eq:classicBellman}
  is such that, for any~$t \in \mathcal{T}$, the truncated sequence
  $\np{\phi_{t}\opt, \dots, \phi_{T{-}1}\opt}$ is an optimal solution
  of Problem~\eqref{pb:classic-Jt} for any initial state~$x_{t}$.
  Thus, according to Definition~\ref{def:timeconsistency}, the sequence
$\np{\phi_{0}\opt, \dots, \phi_{T{-}1}\opt}$ of controls is time
  consistent for the optimization data set~$\mathscr{S}$.

\begin{remark}
\label{rem:tc}
  The notion of time consistency crucially depends on the nature
  of the solutions of the family of optimization
  problem under consideration. As a matter of fact, consider
  Problem~\eqref{pb:classic} and
  its solution~$(\phi_{0}\opt, \dots, \phi_{T{-}1}\opt)$ in terms
  of feedback strategies: as already explained, there is no difficulty
  to apply the truncated sequence~$(\phi_{t}\opt, \dots, \phi_{T{-}1}\opt)$
  to Problem~\eqref{pb:classic-Jt} since this truncated sequence is admissible
  for the problem starting at time~$t$. But consider again Problem~\eqref{pb:classic}
  and its solution~$\va{u}\opt=(\va{u}\opt_{0},\ldots,\va{u}\opt_{T{-}1})$
  in terms of random variables. Problem~\eqref{pb:classic-Jt} is equivalent to
  \begin{subequations}
  \label{pb:classic-t}
  \begin{align}
    \min_{\va{u},\va{x}} \quad
    & \bgesp{\sum_{\tau=t}^{T{-}1}
      L_\tau \np{\va{x}_\tau,\va{u}_\tau,\va{w}_{\tau+1}} + K \np{\va{x}_T}}
      \eqfinv \\
    \text{s.t.} \quad
    & \va{x}_{t} = x_{t} \eqfinv \\
    & \va{x}_{\tau+1} = f_\tau \np{\va{x}_\tau,\va{u}_\tau,\va{w}_{\tau+1}} \eqsepv
      \qquad \forall \tau \in \ic{t,T{-}1} \eqfinv \\
    & \sigma\np{\va{u}_\tau} \subset
      \sigma\np{
      \va{w}_{t+1},\dots,\va{w}_\tau} \eqsepv
      \qquad \forall \tau \in \ic{t,T{-}1}
      \label{pb:classic-t-mes} \eqfinp
  \end{align}
  \end{subequations}
  We note that the truncated subsequence~$(\va{u}\opt_{t},\ldots,\va{u}\opt_{T{-}1})$
  of~$\va{u}\opt$  is not even admissible for
  Problem~\ref{pb:classic-t} as it does not satisfy~\eqref{pb:classic-t-mes}.
  Indeed, each~$\va{u}\opt_{\tau}$ for $\tau \in \ic{t,T{-}1}$ is by construction
  (see Constraint~\eqref{pb:classic-mes}) measurable with respect
  to the~$\sigma$-field $\sigma(\va{w}_{1}$, $\ldots$, $\va{w}_{\tau})$ and thus
  does not satisfy Constraint~\eqref{pb:classic-t-mes}.
\end{remark}

\section{Multistage stochastic optimization with a final constraint in expectation}
\label{sec:constraint}

In \S\ref{ssec:constraint-standard}, we modify the framework studied in
\S\ref{ssec:classic} by adding to Problem~\eqref{pb:classic} a constraint
in expectation involving the final state~$\va{x}_{T}$,
which leads to the optimization problem~\eqref{pb:constraint} below.
In \S\ref{ssec:constraint-equivalent}, we propose a reformulation
of Problem~\eqref{pb:constraint} involving a finite dimensional
state, and an optimization data set (including
the initial state and the level of the expectation constraint)
for which time consistency holds.
In \S\ref{ssec:constraint-dual}, we propose a dual problem formulation
of Problem~\eqref{pb:constraint}, and we illustrate the fact that such
a reformulation is not time consistent.

\subsection{Standard formulation}
\label{ssec:constraint-standard}

We use the notations defined in~\S\ref{ssec:classic}. We consider
a measurable function \( g : \espacea{X}_{T} \to \mathbb{R}_+^{m} \),
For convenience, we denote~$\espacea{Z}_{T}=\mathbb{R}^{m}$.
The stochastic optimization problem starting
at time~$\tzero \in \ic{0,T{-}1} $ with a final constraint in expectation
at time~$T$ is
\begin{subequations}
  \label{pb:constraint}
  \begin{align}
    \min_{\va{u},\va{x}} \quad
    & \bgesp{\sum_{t=\tzero}^{T{-}1}
      L_t \bp{\va{x}_t,\va{u}_t,\va{w}_{t+1}} + K \np{\va{x}_T}}
      \eqfinv \label{pb:constraint-crit} \\
    \text{s.t.} \quad
    & \va{x}_{\tzero} = x_{\tzero} \eqfinv \label{pb:constraint-ini} \\
    & \va{x}_{t+1} = f_t \bp{\va{x}_t,\va{u}_t,\va{w}_{t+1}} \eqsepv
      \qquad\quad\; \forall t\in \ic{\tzero,T{-}1}
      \eqfinv \label{pb:constraint-dyn} \\
    & \sigma\bp{\va{u}_t} \subset
      \sigma\bp{\va{w}_{\tzero{+}1},\dots,\va{w}_t} \eqsepv
      \qquad\; \forall t \in \ic{\tzero,T{-}1}
      \eqfinv \label{pb:constraint-mes} \\
    & \besp{g(\va{x}_{T})} -  b_{\tzero}\leq 0
      \eqfinv
      \label{pb:constraint-exp}
  \end{align}
\end{subequations}
with~$b_{\tzero}\in\mathbb{R}^{m}$. Again, Problem~\eqref{pb:constraint}
is assumed to be well-defined.

Even under the Markovian Assumption~\ref{hyp:Markov}, the presence
of Constraint~\eqref{pb:constraint-exp} makes it difficult
to write a dynamic programming equation for solving
Problem~\eqref{pb:constraint} starting at time~$\tzero=0$.
Indeed Constraint~\eqref{pb:constraint-exp} is not a pointwise
constraint at the final stage~$T$, so that we do not know how
to incorporate it easily in the dynamic programming equation.
In~\S\ref{ssec:constraint-dual}, using an indirect way of proceeding,
we show that it is possible to obtain  an optimal solution
of Problem~\eqref{pb:constraint} starting at time~$\tzero=0$
in terms of feedback strategies depending on the state~$\va{x}_{t}$.
But these feedbacks are implicitly parameterized by both the initial
state~$x_{0}$ and the constraint level~$b_{0}$, so that
they do not satisfy the time consistency property for
a data set in which the parameter at time~$t$ is the initial
state~$x_{t}$ (see~\S\ref{ssec:constraint-dual} for further details).

A partial answer to the question of time consistency
of the solution of Problem~\eqref{pb:constraint} starting
at time~$t_{0}=0$ has been given
in~\cite{Carpentier-Chancelier-Cohen-DeLara-Girardeau:2012}.
Indeed, as detailed in~\cite{Carpentier-Chancelier-Cohen-DeLara-Girardeau:2012},
Problem~\eqref{pb:constraint} can be written in an equivalent way
as a deterministic distributed optimal control problem in
which the state variable is the probability distribution of~$\va{x}_{t}$,
the dynamics of which is given by the Fokker-Planck equation.
This deterministic problem can be solved
by dynamic programming, which thus produces a sequence~$(\Psi_{0}\opt, \dots,
\Psi_{T{-}1}\opt)$ of strategies (with the mapping~$\Psi_{t}$ defined over
probability distributions on~$\espacea{X}_{t}$ and taking values
in~$\espacea{U}_{t}$ \cite{Bertsekas-Shreve:1996}),
which is time consistent for a data set in which
the parameter at time~$t$ is the initial state probability distribution.
But these optimal strategies depend on the specific
value~$b_{0}$ in the right-hand side of the expectation
constraint, and thus have to be recomputed if this value changes.
Moreover, the computation of the Bellman functions involves an infinite
dimensional state, so that it is generally not tractable.

Our goal is to obtain a solution for Problem~\eqref{pb:constraint}
starting at time~$\tzero=0$ which, on the one hand is computable
in practice (that is, involves a finite dimensional state),
and on the other hand is time consistent for a data set
(to be specified) in which the parameter at time~$t$ consists of
both the initial state~$x_{t}$ and the constraint level~$b_{t}$.
More precisely, we want to compute a solution for Problem~\eqref{pb:constraint}
starting at time~$\tzero=0$ which is optimal for
any value of both the initial state~$x_{0}$ in~\eqref{pb:constraint-ini}
and the final constraint level~$b_{0}$ in~\eqref{pb:constraint-exp}.
Moreover, for any~$t\in\ic{0,T{-}1}$, this solution after truncation has
to be a universal solution (Definition~\ref{def:universalsolution})
for the parameters~$(x_{t},b_{t})$
for Problem~\eqref{pb:constraint} starting at~$\tzero=t$.
As already explained in Remark~\ref{rem:tc}, time consistency is not
available for a solution in terms of random variables. We now present
a reformulation of Problem~\eqref{pb:constraint} involving a finite
dimensional state, whose solution in terms of state feedback strategies meets
the goal described in this paragraph.

\subsection{Formulation with martingale-type constraints}
\label{ssec:constraint-equivalent}

Following the same path as in~\cite{Bouchard_SIAM_2009},
but in a discrete time context, we show that
Problem~\eqref{pb:constraint} is equivalent to a multistage
stochastic optimization problem subject to an almost sure contraint
on the final state (see also \cite{Girardeau_These_2010,Granato_These_2012,Pfeiffer_AMO_2018}).
For that purpose, we introduce a new state process
$\va{z}=(\va{z}_{0},\ldots,\va{z}_{T})$ and a new control
process $\va{v}=(\va{v}_{0},\ldots,\va{v}_{T{-}1})$.
The random variables~$\va{z}_{t}$ and~$\va{v}_{t}$ take their values
respectively in spaces~$\espacea{V}_{t}$ and~$\espacea{Z}_{t}$,
all identical to the space~$\espacea{Z}_{T}=\mathbb{R}^{m}$
where~$\mathbb{R}^{m}$ is the codomain of the function~$g$
introduced at the beginning of~\S\ref{ssec:constraint-standard}:
$\espacea{V}_{t}=\espacea{Z}_{t}=\espacea{Z}_{T}=\mathbb{R}^{m}$
for all~$t\in\ic{0,T{-}1}$.
Now, we consider the optimization problem starting at time~$\tzero$
\begin{subequations}
  \label{pb:constraint-equiv}
  \begin{align}
    \min_{(\va{u},\va{v},\va{x},\va{z})} \;
    & \bgesp{\sum_{t=\tzero}^{T{-}1}
      L_{t}(\va{x}_{t},\va{u}_{t},\va{w}_{t+1})+K(\va{x}_{T})}
      \eqfinv \label{pb:constraint-equiv-cost}
      \intertext{subject, for all~$t=\tzero,\dots,T{-}1$, to dynamic constraints}
    & \va{x}_{\tzero} = x_{\tzero} \eqsepv \quad
      \va{x}_{t+1} = f_{t}(\va{x}_{t},\va{u}_{t},\va{w}_{t+1})
      \eqfinv \label{pb:constraint-equiv-dynx} \\
    & \va{z}_{\tzero} = z_{\tzero} \eqsepv \quad\;
      \va{z}_{t+1} = \va{z}_{t} + \va{v}_{t}
      \eqfinv \label{pb:constraint-equiv-dynz}
      \intertext{to measurability constraints}
    & \sigma\bp{\va{u}_{t}} \subset
      \sigma\bp{
      \va{w}_{\tzero{+}1},\dots,\va{w}_t}
      \eqfinv \label{pb:constraint-equiv-mesu} \\
    & \sigma\bp{\va{v}_{t}} \subset
      \sigma\bp{
      \va{w}_{\tzero{+}1},\dots,\va{w}_{t+1}}
      \eqfinv \label{pb:constraint-equiv-mesv}
      \intertext{to martingale-type constraints}
    & \va{v}_{t} \text{ is integrable and }
      \bespc{\va{v}_{t}}
      {\sigma\np{
      \va{w}_{\tzero{+}1},\dots,\va{w}_t}} = 0
      \eqfinv \label{pb:constraint-equiv-mart}
      \intertext{and to almost sure final constraint}
    & g(\va{x}_{T}) - \va{z}_{T} \leq 0
      \eqfinp \label{pb:constraint-equiv-finl}
  \end{align}
\end{subequations}
Note that, in this formulation, the control variable~$\va{u}_{t}$ taken at time~$t$ does not depend on
the noise~$\va{w}_{t+1}$ (\emph{Decision--Hazard} framework),
whereas the control variable~$\va{v}_{t}$, also taken
at time~$t$, does depend on the noise~$\va{w}_{t+1}$
(\emph{Hazard--Decision} framework). At time~$t$, the
martingale-type constraint~\eqref{pb:constraint-equiv-mart}
introduces a coupling between all the realizations
of the decision random variable~$\va{v}_{t}$. In the sequel,
for all $t \in \ic{\tzero{+}1,T}$, we denote by~$\tribu{F}_{\tzero:t}$
the~$\sigma$-field generated by the
sequence~$(\va{w}_{\tzero{+}1},\dots,\va{w}_t)$
of random variables: 
\begin{equation}
  \label{eq:sigma-field-t}
  \tribu{F}_{\tzero:t} =
  \sigma\bp{
    \va{w}_{\tzero{+}1},\dots,\va{w}_t}
  \eqfinp
\end{equation}
By convention, $\tribu{F}_{\tzero:\tzero} =\na{\emptyset, \omeg}$.

\subsubsection{Equivalence with the standard formulation}

The link between the multistage stochastic optimization control
problem~\eqref{pb:constraint-equiv} incorporating a martingale-type
constraint and the initial problem~\eqref{pb:constraint} is given
by the following proposition.

\begin{proposition}
\label{pr:equivalence}
We suppose that the data of the Problems~\eqref{pb:constraint}
and~\eqref{pb:constraint-equiv} are linked by
\begin{equation}
  \label{eq:data-compatibility}
  b_{\tzero} = z_{\tzero} \eqfinp
\end{equation}
Then, Problem~\eqref{pb:constraint} and Problem~\eqref{pb:constraint-equiv}
are equivalent, in the sense that
\begin{itemize}
  \item a solution
    $(\va{u}^{\flat},\va{v}^{\flat},\va{x}^{\flat},\va{z}^{\flat})$
    of Problem~\eqref{pb:constraint-equiv} can be deduced from a solution
    $(\va{u}\opt,\va{x}\opt)$ of Problem~\eqref{pb:constraint},
  \item the two first components $(\va{u}^{\flat},\va{x}^{\flat})$ of a solution
    $(\va{u}^{\flat},\va{v}^{\flat},\va{x}^{\flat},\va{z}^{\flat})$
    of Problem~\eqref{pb:constraint-equiv} is a solution
    of Problem~\eqref{pb:constraint}.
\end{itemize}
\end{proposition}

\begin{proof}
  Let~$(\va{u}\opt,\va{x}\opt)$ be a solution of
  Problem~\eqref{pb:constraint}. We define the random
  processes~$\va{v}\opt$ and~$\va{z}\opt$ by
  \begin{subequations}
    \label{eq:construction-v}
    \begin{align}
      & \va{v}\opt_{t} = \bespc{g(\va{x}\opt_{T})}{\tribu{F}_{\tzero:t+1}} -
        \bespc{g(\va{x}\opt_{T})}{\tribu{F}_{\tzero:t}} \eqsepv
        \quad \forall t \in \ic{\tzero,T{-}1}
        \eqfinv \label{eq:construction-vv}\\
      & \va{z}\opt_{\tzero} = z_{\tzero} \eqsepv \quad
        \va{z}\opt_{t+1} = \va{z}\opt_{t} + \va{v}\opt_{t} \eqsepv
        \qquad\qquad\qquad\; \forall t \in \ic{\tzero,T{-}1}
        \eqfinp \label{eq:construction-vz}
    \end{align}
  \end{subequations}
  The random vector~$\va{v}\opt_{t}$ is well defined
  (hence so is~$\va{z}\opt_{t}$) and is integrable. Indeed,
  as the function \( g : \espacea{X}_{T} \to \mathbb{R}_+^{m} \)
  is assumed nonnegative in~\S\ref{ssec:constraint-standard} and
  using Inequality~\eqref{pb:constraint-exp}, we get
  \( 0 \leq \nesp{g(\va{x}\opt_{T})} \leq b_{\tzero} \).
  As $b_{\tzero} \in \mathbb{R}^{m}$, the random vector~$g(\va{x}\opt_{T})$
  is integrable, hence $\va{v}\opt_{t}$ in~\eqref{eq:construction-vv}
  is the difference between two random vectors in~$\mathbb{R}^{m}$,
  hence is well defined, and is integrable.

  By construction, the two processes~$\va{v}\opt$ and~$\va{z}\opt$ satisfy the constraints
  \eqref{pb:constraint-equiv-dynz}--\eqref{pb:constraint-equiv-mesv}--\eqref{pb:constraint-equiv-mart}.
  Moreover, we have that
  \begin{align*}
    \va{z}\opt_{T}
    & = z_{\tzero} + \sum_{t=\tzero}^{T{-}1} \va{v}\opt_{t}
      \tag{by~\eqref{eq:construction-vz}} \\
    & = z_{\tzero} + g(\va{x}\opt_{T}) - \nesp{g(\va{x}\opt_{T})}
      \tag{by telescoping sum using~\eqref{eq:construction-vv}} \eqfinv
  \end{align*}
  and hence, by~\eqref{pb:constraint-exp} and~\eqref{eq:data-compatibility}, we get that
  \begin{equation*}
    g(\va{x}\opt_{T})-\va{z}\opt_{T} = \besp{g(\va{x}\opt_{T})}-z_{\tzero}
    \leq  b_{\tzero}- z_{\tzero} = 0 \eqfinv
  \end{equation*}
  so that Constraint~\eqref{pb:constraint-equiv-finl}
  is also fulfilled. We deduce that
  $(\va{u}\opt,\va{v}\opt,\va{x}\opt,\va{z}\opt)$
  is admissible for Problem~\eqref{pb:constraint-equiv}.
  Suppose that there would exist a solution
  $(\va{u}^{\flat},\va{v}^{\flat},\va{x}^{\flat},\va{z}^{\flat})$
  of Problem~\eqref{pb:constraint-equiv} with
  a strictly lower cost value than
  $(\va{u}\opt,\va{v}\opt,\va{x}\opt,\va{z}\opt)$.
  From the dynamics~\eqref{pb:constraint-equiv-dynz}, we would have
  \begin{equation*}
    \va{z}^{\flat}_{T} = z_{\tzero} + \sum_{t=\tzero}^{T{-}1} \va{v}^{\flat}_{t}
    \eqfinv
  \end{equation*}
  so that~$\nesp{\va{z}^{\flat}_{T}} = z_{\tzero}$
  by repeated uses of~\eqref{pb:constraint-equiv-mart}.
  Taking the expectation in~\eqref{pb:constraint-equiv-finl}
  would lead thus to
  \begin{equation*}
    \besp{g(\va{x}^{\flat}_{T})} \leq \nesp{\va{z}^{\flat}_{T}} =
    z_{\tzero} = b_{\tzero} \eqfinp
  \end{equation*}
  Then, $(\va{u}^{\flat},\va{x}^{\flat})$ would be admissible
  for Problem~\eqref{pb:constraint} with a strictly lower cost
  value than the optimal solution~$(\va{u}\opt,\va{x}\opt)$,
  which contradicts the assumed optimality of~$(\va{u}\opt,\va{x}\opt)$.
  We conclude that
  $(\va{u}\opt,\va{v}\opt,\va{x}\opt,\va{z}\opt)$
  is an optimal solution of Problem~\eqref{pb:constraint-equiv}.

  \medskip

  Conversely, let
  $(\va{u}^{\flat},\va{v}^{\flat},\va{x}^{\flat},\va{z}^{\flat})$
  be an optimal solution of Problem~\eqref{pb:constraint-equiv}.
  As shown in the first part of the proof,
  we have~$\besp{g(\va{x}^{\flat}_{T})} \leq b_{\tzero}$,
  so that~$(\va{u}^{\flat},\va{x}^{\flat})$ is admissible
  for Problem~\eqref{pb:constraint}.
  Suppose that there would exist a solution $(\va{u}\opt,\va{x}\opt)$
  of Problem~\eqref{pb:constraint} with a strictly lower
  cost value than $(\va{u}^{\flat},\va{x}^{\flat})$.
  Then, the quadruplet
  $(\va{u}\opt,\va{v}\opt,\va{x}\opt,\va{z}\opt)$ obtained
  by constructing~$\va{v}\opt$ and~$\va{z}\opt$
  by~\eqref{eq:construction-v} would give a strictly lower cost
  value for Problem~\eqref{pb:constraint-equiv} than
  $(\va{u}^{\flat},\va{v}^{\flat},\va{x}^{\flat},\va{z}^{\flat})$,
  which would be absurd. We conclude that
  $(\va{u}^{\flat},\va{x}^{\flat})$
  is an optimal solution of Problem~\eqref{pb:constraint}.
\end{proof}

Let us make a few comments about
Problem~\eqref{pb:constraint-equiv}.
\begin{itemize}
\item The nice features of the
  equivalent formulation~\eqref{pb:constraint-equiv}
  of Problem~\eqref{pb:constraint} are double. On the one hand,
  the initial constraint~\eqref{pb:constraint-exp} in expectation is replaced
  by an almost sure constraint~\eqref{pb:constraint-equiv-finl} on the final state, hence
  paving the way to use dynamic programming to solve
  Problem~\eqref{pb:constraint-equiv}.
  On the other hand, the parameter defining the right-hand side
  of the constraint~\eqref{pb:constraint-exp} in expectation
  in formulation~\eqref{pb:constraint} becomes a component~\eqref{pb:constraint-equiv-dynz} of the initial
  state in Problem~\eqref{pb:constraint-equiv}, thus leading to
time consistency as a consequence of dynamic programming.
\item The conditional expectation
  $\bespc{g(\va{x}_{T})}{\tribu{F}_{\tzero:t}}$ can be interpreted as
  the ``perception of the risk constraint~\eqref{pb:constraint-exp}'' at time~$t$.
  From the very definition of $\va{v}_{t}$, we have that
  \begin{equation*}
    \va{v}_{t} = \bespc{g(\va{x}_{T})}{\tribu{F}_{\tzero:t+1}} -
    \bespc{g(\va{x}_{T})}{\tribu{F}_{\tzero:t}} \eqfinv
  \end{equation*}
  from which we deduce that the control $\va{v}_{t}$ corresponds to the
  variation of this perception between time~$t$ and time $t+1$.
  The additional state $\va{z}_{t+1}$ in~\eqref{pb:constraint-equiv-dynz} is thus the cumulative
  variation of the risk contraint perception up to time $t+1$.
  Therefore, this new added state seems to be the minimal
  information which has to be added to the standard state
  in order to recover a dynamic programming principle.
\item Nevertheless, Problem \eqref{pb:constraint-equiv} is rather
  more intricate that Problem \eqref{pb:constraint}:
  \begin{enumerate}
    \item there are additional state and control processes~$\va{z}$
    and~$\va{v}$,
    \item the new control variables~$\va{v}_{t}$ has to be searched
    in the \emph{Hazard--Decision} framework as in~\eqref{pb:constraint-equiv-mesv},
    \item a new expectation constraint~\eqref{pb:constraint-equiv-mart}
    on the controls~$\va{v}_{t}$ appears at each time step.
  \end{enumerate}
\end{itemize}

\subsubsection{Extended dynamic programming equation and time consistency}

The interest of Problem~\eqref{pb:constraint-equiv}
is highlighted by the following theorem.

\begin{theorem}
\label{th:constraint-equivBellman}
Suppose that the primitive noise random process
$\va{w}$ = $(\va{w}_1$, $\ldots$, $\va{w}_T)$ takes a finite number of values,
and that the following induction (Bellman equation)
\begin{subequations}
  \label{eq:constraint-equivBellman}
  \begin{align}
    \mathscr{V}_{T}(x,z) = \;
    & K(x) + \fcara{\na{g(x) - z \leq 0}}(x,z)
      \label{eq:constraint-equivBellman-T} \eqfinv \\
    \mathscr{V}_{t}(x,z) = \;
    & \min_{u} \; \min_{\substack{\sigma\np{\va{v}} \subset \sigma\np{\va{w}_{t+1} } \\
      \nesp{\va{v}}=0}} \;
      \Besp{ L_{t}(x,u,\va{w}_{t+1}) +
      \mathscr{V}_{t+1}\bp{f_{t}(x,u,\va{w}_{t+1}),z + \va{v}} }
      \label{eq:constraint-equivBellman-t}
  \end{align}
\end{subequations}
is well-defined in the sense that all the functions
\( \mathscr{V}_t : \espacea{X}_t\times\espacea{Z}_t \to \RR_+ \cup\na{+\infty}
\) are measurable, for $t\in\ic{0,T}$.

Then, under Assumption~\ref{hyp:Markov},
Problem~\eqref{pb:constraint-equiv} starting at time~$\tzero=0$
can be solved by dynamic programming with associated Bellman
equation~\eqref{eq:constraint-equivBellman}, and its optimal value
is~$\mathscr{V}_{0}(x_{0},z_{0})$.
\end{theorem}

\begin{proof}
The proof of Theorem~\ref{th:constraint-equivBellman} in given
in Appendix~\ref{ann:progdyn-ext}.
\end{proof}

We deduce from Equation~\eqref{eq:constraint-equivBellman} that there
is no loss of optimality in looking for the optimal control~$\va{u}_t$
at time~$t$ of Problem~\eqref{pb:constraint-equiv} as induced by a measurable mapping
$\phi_{t}: \espacea{X}_t\times\espacea{Z}_t \rightarrow \espacea{U}_t$,
and for the optimal control~$\va{v}_{t+1}$ at time~$t$ as induced by a measurable mapping
$\varphi_{t}:\espacea{X}_t\times\espacea{Z}_t\times\espacea{W}_{t+1}\rightarrow\espacea{V}_t$.

Let us embed Problem~\eqref{pb:constraint-equiv} starting at time~$\tzero=0$
in the framework developed in \S\ref{ssect:uctc}.
The finite time span is~$\mathcal{T}=\ic{0,T{-}1}$,
the sequence of parameter sets is
$\na{\espacea{X}_{t}\times\espacea{Z}_{t}}_{t\in\mathcal{T}}$,
the sequences of control spaces are made of two sequences,
$\na{\espacef{U}_{t}}_{t\in\mathcal{T}}$
with~$\espacef{U}_{t}$ the space of measurable mappings
defined on~$\espacea{X}_{t}\times\espacea{Z}_{t}$
and taking values in~$\espacea{U}_{t}$,
and $\na{\espacef{V}_{t}}_{t\in\mathcal{T}}$ with
$\espacef{V}_{t}$ the space of measurable mappings defined
on $\espacea{X}_{t}\times\espacea{Z}_{t}\times\espacea{W}_{t+1}$
and taking values in~$\espacea{V}_{t}$.
Notice that there exists an additional set
$\espacea{X}_{T}\times\espacea{Z}_{T}$, where the final
state of the system takes values, but that this set is not
part of the sequence of parameter sets.
The sequence of cost functions~$\na{J_{t}}_{t\in\mathcal{T}}$, with
\begin{align*}
  J_{t} \; : \;
  & \espacea{X}_{t}\times\espacea{Z}_{t}\times
    \espacef{U}_{t}\times\espacef{V}_{t}\times\dots\times
    \espacef{U}_{T{-}1}\times\espacef{V}_{T{-}1}
    \, \longrightarrow \; \mathbb{R} 
    \eqfinv
\end{align*}
is defined by
\begin{align*}
  J_{t} (x_{t}
  & ,z_{t},\phi_{t},\varphi_{t},\dots, \phi_{T{-}1},\varphi_{T{-}1}) = \;
  \EE \bigg( \sum_{\tau=t}^{T{-}1} L_\tau
  \bp{\va{x}_\tau,\phi_{\tau}(\va{x}_\tau,\va{z}_\tau),\va{w}_{\tau+1}} \\
  &\hphantom{,z_{t},\phi_{t},\varphi_{t},\dots, \phi_{T{-}1},\varphi_{T{-}1}) = \;\EE \bigg( \sum_{\tau=t}^{T{-}1}}
    + K \np{\va{x}_T}
    + \fcara{\na{g(x) - z \leq 0}}(\va{x}_{T},\va{z}_{T}) \bigg)
    \eqfinv \\
  & \text{ with } \;
    \va{x}_{t} = x_{t} \eqsepv
    \va{x}_{\tau+1} =
    f_\tau \bp{\va{x}_\tau,\phi_{\tau}(\va{x}_\tau,\va{z}_\tau),\va{w}_{\tau+1}}
    \eqsepv \text{ for }\tau\in \ic{t,T{-}1} \eqfinv \\
  & \hphantom{\text{ with } \;}
    \va{z}_{t} = z_{t} \eqsepv
    \va{z}_{\tau+1} = \va{z}_\tau +
    \varphi_{\tau}(\va{x}_\tau,\va{z}_{\tau},\va{w}_{\tau+1})
    \eqsepv \text{ for } \tau\in \ic{t,T{-}1} \eqfinv
\intertext{if $\Bespc{\varphi_{\tau}(\va{x}_\tau,\va{z}_{\tau},\va{w}_{\tau+1})}
{\va{x}_{\tau},\va{z}_{\tau}} = 0 \eqsepv
    \forall \tau\in \ic{t,T{-}1}$, and by}
    J_{t} (x_{t},
    & z_{t},\phi_{t},\varphi_{t},\dots, \phi_{T{-}1},\varphi_{T{-}1})
      = + \infty \qquad \textrm{otherwise.}
\end{align*}

From the optimization data set
$\mathscr{E} =\bp{\mathcal{T},
  \na{\espacea{X}_{t}\times\espacea{Z}_{t}}_{t\in\mathcal{T}},
  \na{\espacef{U}_{t}\times\espacef{V}_{t}}_{t\in\mathcal{T}},
  \na{J_{t}}_{t\in\mathcal{T}}}$
and for a given~$t \in \mathcal{T}$, we build, as in
Definition~\ref{def:universalsolution}, the
family of optimization problems $\mathscr{P}_{t}^{\mathscr{E}}=
\ba{\mathscr{P}_{t}^{\mathscr{E}}(x_{t},z_{t})}_{(x_{t},z_{t})
  \in\espacea{X}_{t}\times\espacea{Z}_{t}}$,
with Problem~$\mathscr{P}_{t}^{\mathscr{E}}(x_{t},z_{t})$
being
\begin{equation}
  \label{pb:constraint-equiv-Jt}
  \min_{\substack{(\phi_{t},\dots,\phi_{T{-}1})\in
      \espacef{U}_{t}\times\dots\times\espacef{U}_{T{-}1} \\
      (\varphi_{t},\dots,\varphi_{T{-}1})\in
      \espacef{V}_{t}\times\dots\times\espacef{V}_{T{-}1}}} \;\;
  J_{t} (x_{t},z_{t},\phi_{t},\varphi_{t},\dots,\phi_{T{-}1},\varphi_{T{-}1})
  \eqfinp
\end{equation}
Optimal strategies $(\phi_{0}\opt, \dots, \phi_{T{-}1}\opt)$ and
$(\varphi_{0}\opt, \dots, \varphi_{T{-}1}\opt)$ obtained by solving
for $t=0$ Problem~\eqref{pb:constraint-equiv-Jt} using the dynamic
programming equation~\eqref{eq:constraint-equivBellman} are such that,
for any~$t \in \mathcal{T}$, $(\phi_{t}\opt, \dots, \phi_{T{-}1}\opt)$
and $(\varphi_{t}\opt, \dots, \varphi_{T{-}1}\opt)$ is an optimal
solution of Problem~\eqref{pb:constraint-equiv-Jt} for any initial
state~$(x_{t},z_{t})$.

We deduce that solving Problem~\eqref{pb:constraint-equiv}
dy dynamic programming fully answers the goal of time consistency
enounced at the end of~\S\ref{ssec:constraint-standard}.
Indeed, for all $t\in\mathcal{T}$, the subsequence of optimal
strategies $(\phi_{t}\opt, \dots, \phi_{T{-}1}\opt)$ is
a universal solution for the family
$\ba{\mathscr{P}_{t}^{\mathscr{E}}(x_{t},z_{t})}_{(x_{t},z_{t})
  \in\espacea{X}_{t}\times\espacea{Z}_{t}}$
of optimization problems~\eqref{pb:constraint-equiv-Jt}.
By Proposition~\ref{pr:equivalence},
a solution of Problem~$\mathscr{P}_{t}^{\mathscr{E}}(x_{t},z_{t})$
induces a solution of Problem~\eqref{pb:constraint} starting
at time~$\tzero=t$ with initial state~$x_{t}$ and final constraint
level~$b_{t}=z_{t}$.

\begin{figure}[htb]
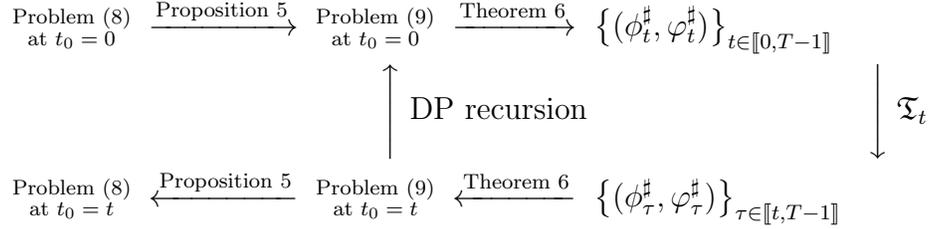

\begin{align*}
  \substack{\text{Problem~\eqref{pb:constraint}} \\ \text{at $\tzero=0$}}
 & \; \xrightarrow{\text{Proposition~\ref{pr:equivalence}}} \;
   \substack{\text{Problem~\eqref{pb:constraint-equiv}}\\\text{at~$\tzero=0$}}
      \; \xrightarrow{\text{Theorem~\ref{th:constraint-equivBellman}}} \;
   \ba{(\phi_{t}\opt,\varphi_{t}\opt)}_{t\in\ic{0,T{-}1}} \\
 & \qquad\qquad\qquad\qquad \Bigg\uparrow \; \text{DP recursion}
   \qquad\qquad\qquad\qquad\quad \Bigg\downarrow  \; \Truncat_{t} \\
  \substack{\text{Problem~\eqref{pb:constraint}}\\\text{at $\tzero=t$}}
 & \; \xleftarrow{\text{Proposition~\ref{pr:equivalence}}} \;
   \substack{\text{Problem~\eqref{pb:constraint-equiv}}\\\text{at $\tzero=t$}}
   \; \xleftarrow{\text{Theorem~\ref{th:constraint-equivBellman}}} \;
   \ba{(\phi_{\tau}\opt,\varphi_{\tau}\opt)}_{\tau\in\ic{t,T{-}1}}
\end{align*}
\caption{Links between the original problem~\eqref{pb:constraint}
         and the extended problem~\eqref{pb:constraint-equiv}}
\end{figure}

Moreover, the Bellman functions defined
by~\eqref{eq:constraint-equivBellman} involve a finite
dimensional state, so that their computation becomes tractable.
We will provide a numerical illustration of this last point in Sect.~\ref{sect:numer}.

\subsection{Formulation with dualized constraint}
\label{ssec:constraint-dual}

We finish this section by presenting a way to solve
Problem~\eqref{pb:constraint} using Lagrangian duality, and we show that
the dualized problem does \emph{not} display time consistency.

\subsubsection{Dualized formulation}

By dualizing the expectation constraint~\eqref{pb:constraint-exp}
in Problem~\eqref{pb:constraint} with a given (fixed)
multiplier~$\lambda\in\mathbb{R}^{m}$, we obtain the following
problem\footnote{%
  For the sake of simplicity, we suppose that the function \( g : \espacea{X}_{T} \to \mathbb{R}_+^{m} \),
introduced at the beginning of~\S\ref{ssec:constraint-standard}, is bounded to
ensure integrability in~\eqref{pb:constraint-dual_criterion}
\label{ft:g_bounded}}
\begin{subequations}
  \label{pb:constraint-dual}
  \begin{align}
    \min_{\va{u},\va{x}} \quad
    & \bgesp{\sum_{t=0}^{T{-}1}
      L_t \bp{\va{x}_t,\va{u}_t,\va{w}_{t+1}} +
      K \np{\va{x}_T} + \lambda \cdot g\np{\va{x}_{T}}}
      \eqfinv
   \label{pb:constraint-dual_criterion}
   \\
    \text{s.t.} \quad
    & \va{x}_{0} = x_{0} \eqfinv \\
    & \va{x}_{t+1} = f_t \bp{\va{x}_t,\va{u}_t,\va{w}_{t+1}}
      \eqsepv \forall t=0, \dots, T{-}1 \eqfinv \\
    & \sigma\bp{\va{u}_t} \subset
      \sigma\bp{
      \va{w}_{1},\dots,\va{w}_t}
      \eqsepv \forall t = 0, \dots, T{-}1 \eqfinp
  \end{align}
By weak duality, we have that the optimal value of this problem is
a lower bound of the optimal value of Problem~\eqref{pb:constraint} for
any value~$\lambda$. In some cases, Problems~\eqref{pb:constraint-dual}
and~\eqref{pb:constraint} are equivalent as specified
by the following theorem.
\end{subequations}

\begin{theorem}
\label{th:pb:constraint-dual}
Assume that there exists~$\lambda_{0}\opt \in \mathbb{R}^{m}$ such that
Problem~\eqref{pb:constraint-dual} with~$\lambda = \lambda_{0}\opt$ admits
a solution~$(\va{u}\opt,\va{x}\opt)$.
We denote by~$\overline{b}_{0}^{\lambda_{0}\opt}$ the expectation
of~$g(\va{x}\opt_{T})$ (which exists by Footnote~\ref{ft:g_bounded})
\begin{equation*}
\overline{b}_{0}^{\lambda_{0}\opt} = \nesp{g(\va{x}\opt_{T})} \eqfinv
\end{equation*}
and we assume that the constraint level~$b_{0}$ in Problem~\eqref{pb:constraint}
is such that~$b_{0} \in B_{0}^{\lambda_{0}\opt}$, with
\begin{multline*}
  B_{0}^{\lambda_{0}\opt} = \big\{ b\in\mathbb{R}^{m} \;\text{ s.t. }\;
  b_{i} = \np{\overline{b}_{0}^{\lambda_{0}\opt}}_{i}
  \text{ if } \np{\lambda_{0}\opt}_{i} > 0 \\
  \;\text{ and }\;
  b_{i} \geq \np{\overline{b}_{0}^{\lambda_{0}\opt}}_{i}
  \text{ if } \np{\lambda_{0}\opt}_{i} = 0
  \eqsepv \forall i=1,\dots,m \big\}
  \eqfinp
\end{multline*}
Then, the solution $(\va{u}\opt,\va{x}\opt)$ of Problem~\eqref{pb:constraint-dual}
is also a solution of Problem~\eqref{pb:constraint}.
\end{theorem}
\begin{proof}
The result is a direct consequence of the extension
of Everett's Theorem given in Appendix~\ref{ann:everett}.
\end{proof}

\begin{remark}
  The Everett argument that has been used here can be
  replaced by a more binding duality argument, namely,
  Problem~\eqref{pb:constraint} admits a saddle
  point~$(\va{u}\opt$, $\va{x}\opt$, $\lambda_{0}\opt)$ when
  dualizing Constraint~\eqref{pb:constraint-exp}.
\end{remark}

\subsubsection{Discussion about time consistency}

Since Problem~\eqref{pb:constraint-dual} falls within
the standard dynamic programming framework, there is
no loss of optimality to look for the optimal
controls~$\va{u}\opt_{t}$ of Problem~\eqref{pb:constraint-dual},
and hence of Problem~\eqref{pb:constraint},
as feedback strategies~$\phi\opt_{t} : \espacea{X}_t\rightarrow \espacea{U}_t$
depending on the state variable~$\va{x}_t$.

However, we do not claim that the optimal feedbacks $\phi\opt_{t}$
obtained by this argument have any specific properties in terms
of time consistency.
Indeed, assume as in Theorem~\ref{th:pb:constraint-dual} that there
exists a~$\lambda_{0}\opt$ such that~$b_{0} \in B_{0}^{\lambda_{0}\opt}$.
The parameter~$\lambda_{0}\opt$ implicitly depends on both the initial
condition~$x_{0}$ \emph{and} the constraint level~$b_{0}$, so that the
optimal feedbacks~$\phi\opt_{t}$ of Problem~\eqref{pb:constraint-dual},
which are parameterized by~$\lambda_{0}\opt$, are accordingly implicitly
parameterized by the pair~$(x_{0},b_{0})$ and therefore do not satisfy
the property of time consistency. Moreover, if we write an optimization
problem similar to Problem~\eqref{pb:constraint-dual} starting at
an initial time~$t>0$ with this value~$\lambda_{0}\opt$
\begin{subequations}
  \label{pb:constraint-dual-t}
  \begin{align}
    \min_{\va{u},\va{x}} \quad
    & \bgesp{\sum_{\tau=t}^{T{-}1}
      L_\tau \bp{\va{x}_\tau,\va{u}_\tau,\va{w}_{\tau+1}} +
      K \np{\va{x}_T} + \lambda_{0}\opt \cdot g\np{\va{x}_{T}}}
      \eqfinv \\
    \text{s.t.} \quad
    & \va{x}_{t} = x_{t} \eqfinv \\
    & \va{x}_{\tau+1} = f_\tau \bp{\va{x}_\tau,\va{u}_\tau,\va{w}_{\tau+1}}
      \eqsepv \forall \tau \in \ic{t,T{-}1} \eqfinv \\
    & \sigma\bp{\va{u}_\tau} \subset
      \sigma\bp{
      \va{w}_{t+1},\dots,\va{w}_\tau}
      \eqsepv \forall \tau = t, \dots, T{-}1 \eqfinv
  \end{align}
\end{subequations}
there is no reason that the optimal solution
$(\va{x}\opt_{t},\dots,\va{x}\opt_{T})$
of Problem~\eqref{pb:constraint-dual-t} satisfies
the constraint $\nesp{g(\va{x}\opt_{T})} \leq b_{0}$,
that is, there is no reason to satisfy
the relation~$b_{0} \in B_{t}^{\lambda_{0}\opt}$.
Of course, it may exists some~$\lambda_{t}\opt$ such that
$b_{0} \in B_{t}^{\lambda_{t}\opt}$, but
usually~$\lambda_{t}\opt\neq\lambda_{0}\opt$.

Thus, the sequence~$(\phi_{0}\opt, \dots, \phi_{T{-}1}\opt)$
of controls obtained by the dual approach, that is, by solving
Problem~\eqref{pb:constraint-dual}, is not time consistent (in the sense of
Definition~\ref{def:timeconsistency}).

\section{Numerical experiments \label{sect:numer}}

We illustrate numerically whether time consistency holds true or not
on a simple dam management problem developed
in~\S\ref{A_dam_management_problem}.
In~\S\ref{Resolution_of_the_problem_with_dualized__expectation_constraint},
we provide a numerical resolution of the problem with dualized  expectation
constraint (as seen in~\S\ref{ssec:constraint-dual}).
In~\S\ref{Resolution_by_extended_dynamic_programming},
we provide a numerical resolution by extended dynamic programming
(as seen in~\S\ref{ssec:constraint-equivalent}).

\subsection{A dam management problem}
\label{A_dam_management_problem}

We consider here a basic dam model for a management problem.
Let $T>0$ denote a positive integer (horizon) and
$ \ic{0,T} $ be the optimization time span,
and let $ \bp{\omeg,\trib,\prbt}$ be a probability space.
For any time~$t$ in~$\ic{0,T{-}1} $, we consider
the following real valued random variables:
\begin{itemize}
\item $\va{x}_{t}$, the \emph{water storage volume}
  in the dam at the beginning of time interval $[t,t+1)$,
\item $\va{u}_{t}$, the \emph{decided amount of water
    to be turbinated} during time interval $[t,t+1)$,
  set at the beginning of the time interval $[t,t+1)$,
  and constrained to belong to
an interval $\nc{\underline{u},\overline{u}}$,
\item $\va{w}_{t+1}$, the \emph{amount of water inflow}
  in the dam during time interval $[t,t+1)$.
\end{itemize}

Let $\underline{x}$ (resp. $\overline{x}$) denotes the minimum
(resp. maximum) water volume of the dam, and let $x_{0}$ be
the dam volume at time~$0$. The decision~$\va{u}_{t}$
can be implemented only if there is enough water in the dam,
that is, the turbinated water during a time interval cannot
exceed the quantity of water present in the dam. Then,
the \emph{real amount of turbinated water} during
the time interval~$[t,t+1)$ is
\begin{equation*}
  \widetilde{\va{u}}_{t+1}= \min
  \ba{\va{u}_{t},\va{x}_{t}+\va{w}_{t+1}-\underline{x}}
  \eqfinp
\end{equation*}
The maximal dam volume~$\overline{x}$ is taken into account
by accepting reservoir overflow: if the forthcoming water
volume $\va{x}_{t}+\va{w}_{t+1}-\widetilde{\va{u}}_{t}$
is greater than $\overline{x}$, then the dam water surplus
$\va{x}_{t}+\va{w}_{t+1}-\widetilde{\va{u}}_{t}-\overline{x} $
spills out. The dam dynamics is written accordingly:
\begin{align*}
  \va{x}_{0} \;\;
  & = x_{0} \eqfinv \\
  \va{x}_{t+1}
  & = \min\ba{\overline{x},
    \va{x}_{t}-\widetilde{\va{u}}_{t+1}+\va{w}_{t+1}}
                   = \min\ba{\overline{x},
    \max\na{\underline{x},\va{x}_{t}-\va{u}_{t}+\va{w}_{t+1}}}
    \eqfinp
\end{align*}
The turbinated water during the time interval~$[t,t+1)$ produces
electricity which is sold at a given price~$p_{t}$. We assume that
the sequence $(p_{0},\dots,p_{T{-}1})$ of prices  is deterministic.
The dam revenue to be maximized is thus
\begin{equation*}
  \bgesp{\sum_{t=0}^{T{-}1} p_{t} \widetilde{\va{u}}_{t}} \eqfinp
\end{equation*}
Compared to classical dam management model, we do not add
an explicit final cost, but we rather constrain the final
level of water in the dam at the end of the time span
by a risk contraint. Indeed, we consider a probability
constraint on the dam water volume at final time~$T$,
namely
\begin{equation*}
  \bprob{\va{x}_{T}\geq \ell} \geq \pi \eqfinv
\end{equation*}
where the water level $\ell$ and the probability level $\pi$
are given real numbers. Ultimately, the problem we have to
solve is
\begin{subequations}
  \label{eq:dam}
  \begin{align}
    \min_{\va{u},\va{x}} \;\;
    & \bgesp{\sum_{t=0}^{T{-}1} - p_{t} \cdot
      \min\ba{\va{u}_{t},\va{x}_{t}+\va{w}_{t+1}-\underline{x}}}
      \label{eq:dam-cost} \eqfinv \\
    \text{s.t. } \quad
    & \va{x}_{0} = x_{0} \label{eq:dam-dynam-x0} \eqfinv \\
    & \va{x}_{t+1} =
      \min\ba{\overline{x},
      \max\na{\underline{x},\va{x}_{t}-\va{u}_{t}+\va{w}_{t+1}}}
      \label{eq:dam-dynam-xt} \eqfinv \\
    & \underline{u} \leq \va{u}_{t} \leq \overline{u}
      \label{eq:dam-dynam-ut}  \eqfinv \\
    & \sigma\bp{\va{u}_{t}} \subset
      \sigma\bp{
      \va{w}_{1},\dots,\va{w}_t}
      \label{eq:dam-measure} \eqfinv \\
    & \bprob{\va{x}_{T}\geq \ell} \geq \pi
      \label{eq:dam-prob}  \eqfinp
  \end{align}
\end{subequations}
We recall that the probability constraint~\eqref{eq:dam-prob}
can be rewritten as an expectation constraint
\begin{equation}
  \label{eq:dam-exp}
  \pi - \besp{\findi{\mathbb{R}^{+}}(\va{x}_{T}-\ell)} \leq 0
  \eqfinv
\end{equation}
where $\findi{\mathbb{R}^{+}} : \mathbb{R} \rightarrow \mathbb{R}$
is the Heaviside step function:
\begin{equation*}
  \findi{\mathbb{R}^{+}}(y) =
  \begin{cases}
    0 & \text{ if } y < 0 \\
    1 & \text{ if } y \geq 0
  \end{cases}
  \eqfinp
\end{equation*}

We assume that the water inflows
$\va{w}_{1},\dots,\va{w}_{T}$ are independent random
variables with a known probability distribution on the interval
$\nc{\underline{w},\overline{w}}$.

To numerically solve Problem~\eqref{eq:dam},
we use the following parameter values:
final time $T=12$;
initial state $x_{0}=10$;
state bounds $\nc{\underline{x},\overline{x}}=\nc{0,20}$;
control bounds $\nc{\underline{u},\overline{u}}=\nc{0,3}$;
noise bounds $\nc{\underline{w},\overline{w}}=\nc{0,4}$;
price sequence $p=\np{10, 10, 10, 8, 6, 4, 4, 4
   , 4, 6, 8, 10}$;
final dam water level $\ell = 10$;
required probability level $\pi = 0.9$.
Moreover, we assume that the variables~$\va{x}_{t}$,
$\va{u}_{t}$ and $\va{w}_{t}$ take discrete values
within their respective bounds, with respective
discretization steps equal to~$0.1$, $0.3$ and $0.2$.
The optimization problem thus corresponds to the control
of a discrete state space Markov chain.
The discrete probability distribution of each random variable~$\va{w}_{t}$ is uniform.
We represent on Figure~\ref{fig:trajectoires-bruit-prix}
some selected trajectories of the noise process
$(\va{w}_{1},\dots,\va{w}_{T})$ and the sequence of
the deterministic prices~$p$. These noise trajectories
are used in the sequel to illustrate the behavior of
the different optimization algorithms.

\begin{figure}[hbtp]
  \begin{center}
    \mbox{\includegraphics[scale=\figscale]{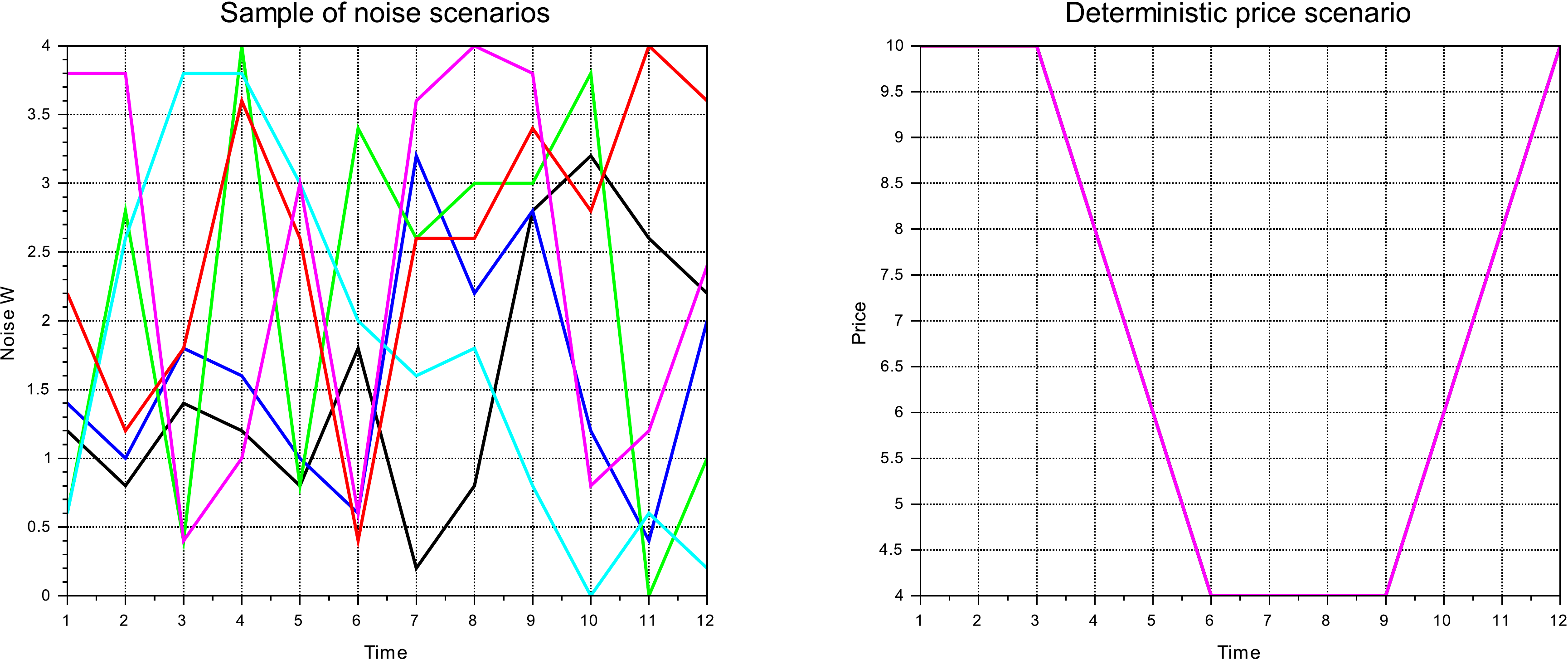}}
    \caption{\label{fig:trajectoires-bruit-prix}Samples of the noise process~$\mathbf{W}$ (left)
      and price scenario~$p$ (right)}
  \end{center}
\end{figure}

\subsection{Resolution of the problem with dualized  expectation constraint}
\label{Resolution_of_the_problem_with_dualized__expectation_constraint}

The method used here to solve Problem~\eqref{eq:dam} is based on the
duality argument described in~\S\ref{ssec:constraint-dual}.
The probability constraint \eqref{eq:dam-prob} is dualized
with associated multiplier $\lambda$, which leads, in the cost
function \eqref{eq:dam-cost}, to an added term of the form:
\begin{equation*}
  \lambda \: \besp{\pi-\findi{\mathbb{R}^{+}}(\va{x}_{T}-\ell)}
  \eqfinp
\end{equation*}
Instead of guessing the parameter value~$\lambda$ leading
to the constraint level~$\ell$, we solve the dual problem
\begin{equation}
  \max_{\lambda \geq 0} \;\; \varphi(\lambda) \eqfinv
  \label{pb:dual}
\end{equation}
with
\begin{subequations}
  \begin{align}
    \varphi(\lambda) \; = \;
    & \min_{\va{u},\va{x}}
      \bgesp{\sum_{t=0}^{T{-}1} - p_{t} \cdot \widetilde{\va{u}}_{t}
      \; - \; \lambda \cdot \findi{\mathbb{R}^{+}}(\va{x}_{T}-\ell) }
      \eqfinv \\
    & \text{under constraints
      \eqref{eq:dam-dynam-x0}--\eqref{eq:dam-dynam-xt}--\eqref{eq:dam-dynam-ut}--\eqref{eq:dam-measure}.}
  \end{align}
  \label{fct:dual}
\end{subequations}


The Uzawa algorithm consists in maximizing the dual
function~$\varphi$ using a projected gradient algorithm.
At iteration~$k$ of the algorithm, knowing the value
$\lambda^{(k)}$ of the multiplier, we perform the three 
following steps.
\begin{itemize}
\item
  Compute~$\varphi(\lambda^{(k)})$, that is,
  solve the minimization problem~\eqref{fct:dual}
  with~$\lambda=\lambda^{(k)}$; this minimization
  is performed using dynamic programming
  (1-dimensional state variable), hence
  furnishing optimal feedbacks~$\phi^{(k+1)}_{t}$
  and optimal state variables~$\va{x}^{(k+1)}_{t}$.
\item
  Compute the probability distribution of the random
  variable~$\va{x}^{(k+1)}_{T}$ by integrating the
  dynamics~\eqref{eq:dam-dynam-x0}-\eqref{eq:dam-dynam-xt}
  using the optimal feedbacks~$\phi^{(k+1)}_{t}$, and then compute
  the value of the constraint $\nprob{\va{x}_{T}^{(k+1)}\geq \ell}$ .
\item Update the multiplier $\lambda$ by
  a projected gradient step:
  \begin{equation*}
    \lambda^{(k+1)} = \proj{\mathbb{R}^{+}}
    {\lambda^{(k)}+\rho\bp{\pi-\nprob{\va{x}^{(k+1)}_{T}\geq \ell}}}
    \eqfinp
  \end{equation*}
\end{itemize}
For the problem under consideration, and despite the
potential nonconvexity induced by the final probability
constraint~\eqref{eq:dam-prob}, the Uzawa algorithm
converges in about 10 iterations, leading to an optimal
multiplier value~$\lambda\opt = 62.405$.
Once the algorithm has converged, we obtain the optimal
feedback sequence~$\na{\phi_{t}\opt}_{t=0,\dots,T{-}1}$
by solving Problem~\eqref{fct:dual} by dynamic programming
with~$\lambda=\lambda\opt$.
Then, we simulate the dynamics of the dam
along some noise trajectories using these optimal feedbacks
$\phi_{t}\opt$. The results given in
Table~\ref{tab:uzawa} are obtained by simulating 10,000
noise trajectories, and illustrates the adequacy between
optimization and simulation.

\begin{table}[htb]
  \begin{center}
    \begin{tabular}{|l||l|}
      \hline
      \textbf{Uzawa optimization}      & \textbf{Monte Carlo simulation} \\
      \hline\hline
      Bellman value at~$t=0$: $-188.90$ & Monte Carlo cost: $-188.94$     \\
      \hline
      Required probability: $0.900$     & Estimated probability: $0.903$  \\
      \hline
    \end{tabular}
  \end{center}
  \caption{Optimization and simulation for the duality method}
  \label{tab:uzawa}
\end{table}

On Figure~\ref{fig:trajectoires-complete-uzawa}, we
represent the dam water level and control trajectories over~$[0,T]$
obtained by simulating with the optimal feedbacks~$\phi_{t}\opt$
along the noise trajectories depicted on
Figure~\ref{fig:trajectoires-bruit-prix}.
We observe that the optimization ``gives up'' for certain trajectories
(the lowest one to the left of Figure~\ref{fig:trajectoires-complete-uzawa})
to reach the final level~$\ell$ appearing in the constraint in probability:
we turbine as much water as possible, leaving the state evolve towards
the minimum level~$\underline{x}$. This observation is conform to the expected
behavior of an optimization problem with a probability constraint.

\begin{figure}[hbtp]
  \begin{center}
    \includegraphics[scale=\figscale]{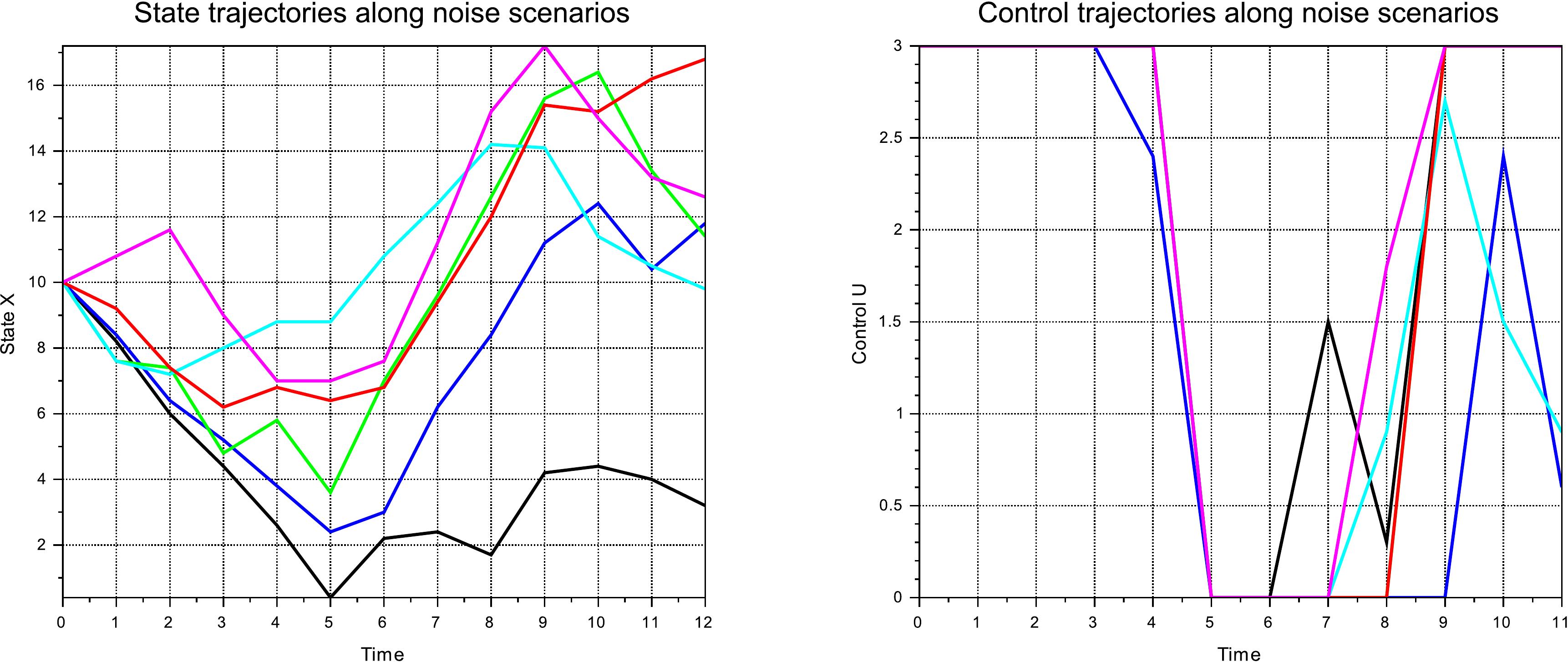}
    \caption{\label{fig:trajectoires-complete-uzawa}Optimal state~$\mathbf{X}$ and control~$\mathbf{U}$
      trajectories over~$[0,T]$ obtained by the duality method}
  \end{center}
\end{figure}

Finally, we can use the optimal feedbacks~$\phi\opt_{t}$
to simulate the dam starting at any initial time~$t_{i}>0$
from any given initial state~$x_{t_{i}}$. For example starting
at time~$t_{i}=3$ from the initial state~$x_{t_{i}}=5$
and simulating the optimal feedbacks~$\phi\opt_{t}$ along
10,000 scenarios leads to the results given
in Table~\ref{tab:uzawa-later}. As expected, the final
constraint level reached in this last simulation is not
equal to the required level~$\pi=0.9$, which illustrates
that time consistency does not hold true for the problem
formulation with dualized constraint.

\begin{table}[htb]
  \begin{center}
    \begin{tabular}{|l|}
      \hline
      \textbf{Monte Carlo simulation} \\
      \hline\hline
      Monte Carlo cost: $-92.04$      \\
      \hline
      Estimated probability: $0.833$  \\
      \hline
    \end{tabular}
  \end{center}
  \caption{Simulation for the duality method starting
    from~$t_{i}=3$ and~$x_{t_{i}}=5$}
  \label{tab:uzawa-later}
\end{table}


\subsection{Resolution by extended dynamic programming}
\label{Resolution_by_extended_dynamic_programming}

We now use the equivalent formulation of
Problem~\eqref{eq:dam} incorporating an additional
state process $\va{z}=(\va{z}_{0},\ldots,\va{z}_{T})$,
an additional control process
$\va{v}=(\va{v}_{0},\ldots,\va{v}_{T{-}1})$
and an almost sure contraint on the final state. As it
has been explained in~\S\ref{ssec:constraint-equivalent},
the expression of the new dynamics, here 1-dimensional, is
\begin{equation*}
  \va{z}_{0} = z_{0} \eqsepv \quad\;
  \va{z}_{t+1} = \va{z}_{t} + \va{v}_{t} \eqfinv
\end{equation*}
and the form of the final constraint is
\begin{equation*}
  - \findi{\mathbb{R}^{+}}(\va{x}_{T}-\ell) - \va{z}_{T} \leq 0
  \eqfinp
\end{equation*}
The equivalent problem for the case study under consideration is
\begin{subequations}
  \begin{align}
    \min_{\va{u},\va{v},\va{x},\va{z}} \;\;
    & \bgesp{\sum_{t=0}^{T{-}1} - p_{t} \cdot
      \min\ba{\va{u}_{t},\va{x}_{t}+\va{w}_{t+1}-\underline{x}}}
      \label{eq:dam-ext-cost} \eqfinv \\
    \text{s.t. } \quad
    & \va{x}_{0} = x_{0} \eqsepv \;
      \va{x}_{t+1} =
      \min\ba{\overline{x},
      \max\na{\underline{x},\va{x}_{t}-\va{u}_{t}+\va{w}_{t+1}}}
      \label{eq:dam-ext-dynam-x} \eqfinv \\
    & \va{z}_{0} = z_{0} \eqsepv \;\;
      \va{z}_{t+1} = \va{z}_{t} + \va{v}_{t}
      \label{eq:dam-ext-dynam-z} \eqfinv \\
        & \underline{u} \leq \va{u}_{t} \leq \overline{u}
      \label{eq:dam-ext-bounds-u}  \eqfinv \\
    & \sigma\bp{\va{u}_{t}} \subset
      \sigma\bp{
      \va{w}_{1},\dots,\va{w}_t}
      \label{eq:dam-ext-measure-u} \eqfinv \\
    & \sigma\bp{\va{v}_{t}} \subset
      \sigma\bp{
      \va{w}_{1},\dots,\va{w}_{t+1}}
      \label{eq:dam-ext-measure-v} \eqfinv \\
    & \Bespc{\va{v}_{t}}
      {\sigma\bp{
      \va{w}_{1},\dots,\va{w}_t}} = 0
      \label{eq:dam-ext-mart} \eqfinv \\
    & \findi{\mathbb{R}^{+}}(\va{x}_{T}-\ell) + \va{z}_{T} \geq 0
      \eqfinp \label{eq:dam-ext-finl}
  \end{align}
  \label{eq:dam-ext}
\end{subequations}
From Proposition~\ref{pr:equivalence}, we have that
Problems~\eqref{eq:dam} and~\eqref{eq:dam-ext} are
equivalent under the condition
\begin{equation*}
  z_{0} = - \pi \eqfinp
\end{equation*}
Moreover, the special form~\eqref{eq:dam-ext-finl} of the final constraint makes it
possible to bound the variables~$\va{v}_{t}$ and~$\va{~z}_{t}$.
Indeed, from the proof of Proposition~\ref{pr:equivalence},
we deduce from the expression~\eqref{eq:construction-vv}
of the optimal control~$\va{v}_{t}\opt$
with~$g=\findi{\mathbb{R}^{+}}$ that it is sufficient
to search for the control~$\va{v}_{t}$ in~$[-1,1]$.
Moreover, the optimal state~$\va{z}_{t}\opt$ being obtained
by a telescoping sum, it is sufficient to search for
the state~$\va{z}_{t}$ in~$[z_{0}-1,z_{0}+1]$.
Problem~\eqref{eq:dam-ext} can be solved by dynamic programming
with the extended state variable~$(\va{x}_{t},\va{z}_{t})$,
which corresponds to a dynamic programming equation with
a 2-dimensional state variable.
Then, we simulate the dynamics of the dam --- using
the same 10,000 noise trajectories
as those previously used to obtain Table~\ref{tab:uzawa} ---
with the optimal feedbacks given by dynamic programming
with the extended state variable~$(\va{x}_{t},\va{z}_{t})$.
The associated results  are given in Table~\ref{tab:extDP}.
We observe a good adequacy between optimization and
simulation, and we observe that the costs are pretty much
identical between Table~\ref{tab:uzawa} and Table~\ref{tab:extDP}.


\begin{table}[htb]
  \begin{center}
    \begin{tabular}{|l||l|}
      \hline
      \textbf{Extended Dynamic Programming} & \textbf{Monte Carlo Simulation} \\
      \hline\hline
      Bellman value at~$t=0$: $-188.67$     & Monte Carlo cost: $-187.47$     \\
      \hline
      Initial state~$z_{0}$: $-0.900$       & Estimated probability: $0.896$  \\
      \hline
    \end{tabular}
  \end{center}
  \caption{Optimization and simulation for the extended dynamic programming method}
  \label{tab:extDP}
\end{table}

Some simulation trajectories are represented on
Figure~\ref{fig:trajectoires-completeXU-extdp} and
Figure~\ref{fig:trajectoires-completeZV-extdp}.
Figure~\ref{fig:trajectoires-completeXU-extdp} gives the same information
(dam water level and control trajectories over~$[0,T]$) as the
one presented for the duality based algorithm in~\S\ref{Resolution_of_the_problem_with_dualized__expectation_constraint},
whereas Figure~\ref{fig:trajectoires-completeZV-extdp} depicts trajectories
of the optimal state process~$\va{z}$ and the optimal control process~$\va{v}$.
We observe that the results depicted on
Figure~\ref{fig:trajectoires-complete-uzawa} (duality method)
and on Figure~\ref{fig:trajectoires-completeXU-extdp}
(extended dynamic programming method) are very close
(with tiny differences induced by numerical resolution), which illustrates the
equivalence between Problem~\eqref{eq:dam} and Problem~\eqref{eq:dam-ext}.

\begin{figure}[hbtp]
  \begin{center}
    \includegraphics[scale=\figscale]{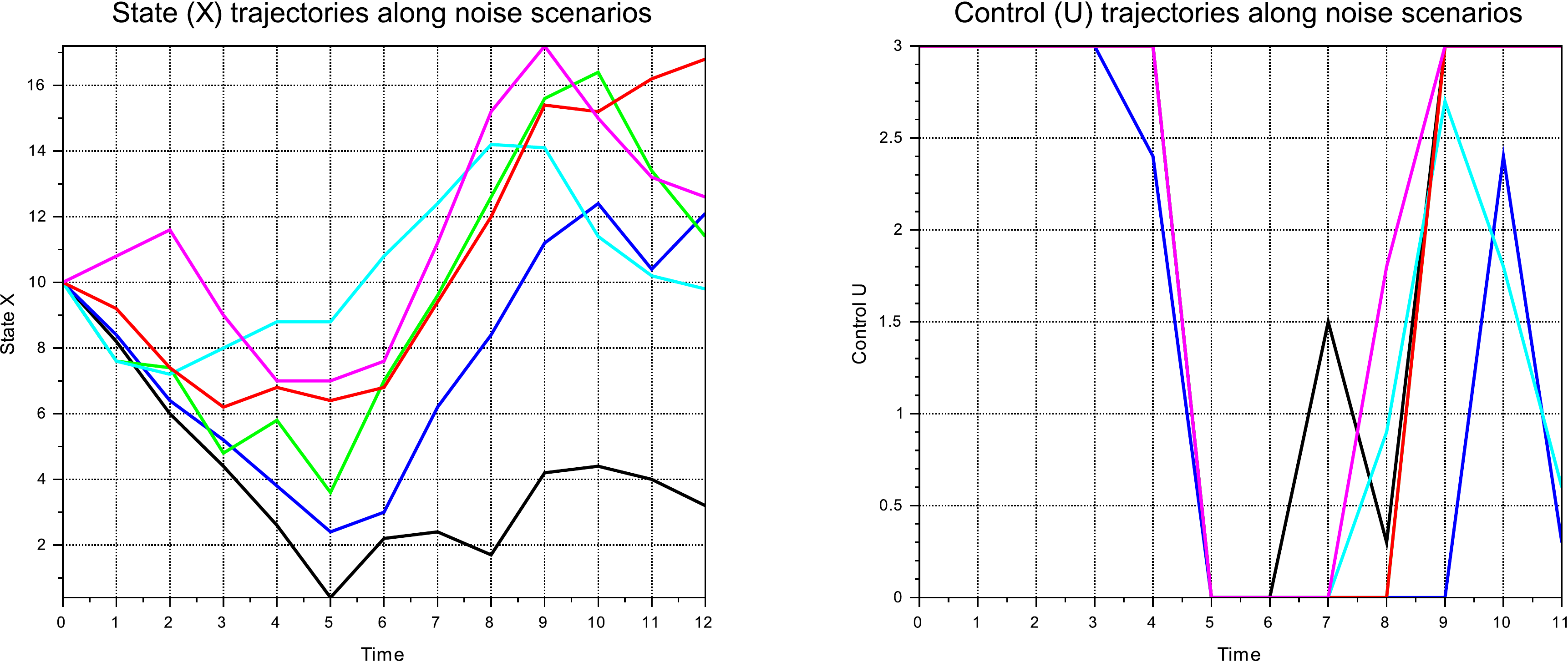}
    \caption{\label{fig:trajectoires-completeXU-extdp}Optimal state~$X$ and control~$U$ trajectories
      over~$[0,T]$ obtained by the extended dynamic programming method}
  \end{center}
\end{figure}
\begin{figure}[hbtp]
  \begin{center}
    \includegraphics[scale=\figscale]{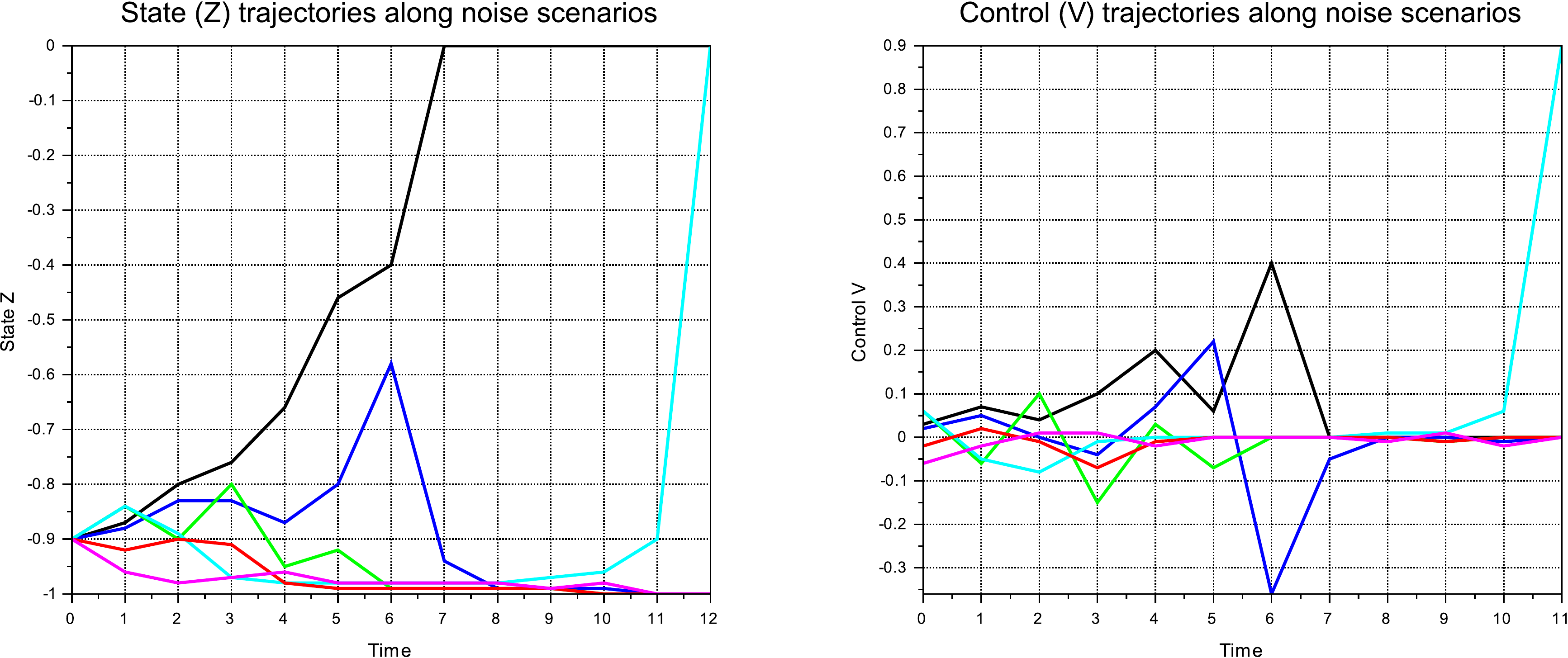}
    \caption{\label{fig:trajectoires-completeZV-extdp}Optimal state~$Z$ and control~$V$ trajectories
      over~$[0,T]$ obtained by the extended dynamic programming method}
  \end{center}
\end{figure}

Finally, we can use the optimal feedbacks obtained
when solving Problem~\eqref{eq:dam-ext} by
dynamic programming  to simulate the dam starting at any
initial time~$t_{i}>0$ from any given initial state
$(x_{t_{i}},z_{t_{i}})$. For example starting at time
$t_{i}=3$ from the initial state~$(x_{t_{i}},z_{t_{i}})=(5,-0.9)$ --- using
the same values than those used in~\S\ref{Resolution_of_the_problem_with_dualized__expectation_constraint} ---
and simulating the optimal feedbacks along 10,000 scenarios
leads to the results given in Table~\ref{tab:extDP-later}:
the final probability level to be reached is by construction
equal to~$\pi=0.9$, and the Monte Carlo simulation induces
a very similar level of probability, which numerically illustrates
that the time consistency property is fulfilled.
The associated simulation trajectories are represented
on Figure~\ref{fig:trajectoires-partielXU-extdp}, using
the same noises trajectories as those that had been used to obtain
Figure~\ref{fig:trajectoires-complete-uzawa}.
\begin{table}[htbp]
  \begin{center}
    \begin{tabular}{|l||l|}
      \hline
      \textbf{Extended Dynamic Programming}       & \textbf{Monte Carlo Simulation} \\
      \hline\hline
      Bellman value at~$t=3$: $-87.78$            & Monte Carlo cost: $-87.71$      \\
      \hline
      Initial state~$(x_{3},z_{3})$: $(5,-0.900)$ & Estimated probability: $0.902$  \\
      \hline
    \end{tabular}
  \end{center}
  \caption{Optimization and simulation for the extended dynamic programming method}
  \label{tab:extDP-later}
\end{table}

\begin{figure}[hbtp]
  \begin{center}
    \includegraphics[scale=\figscale]{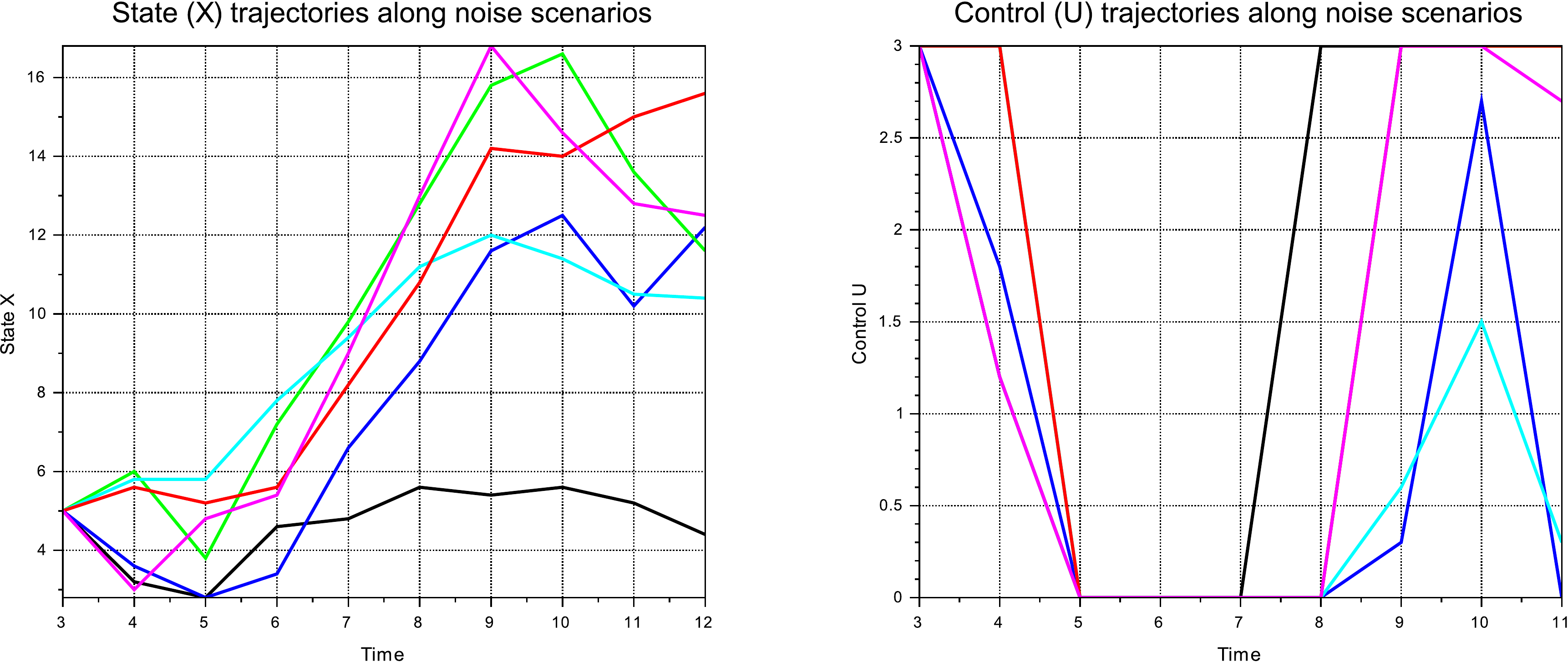}
    \caption{\label{fig:trajectoires-partielXU-extdp}Optimal state~$X$ and control~$U$ trajectories
      over~$[3,T]$ obtained by the extended dynamic programming method}
  \end{center}
\end{figure}

\section{Conclusion}
\label{sect:conclusion}

In this paper, we have proposed a formal definition of time consistency
for families of optimization problems, by introducing the notion of
universal solution. With this, we have shown that --- for the class
of problems where risk is modeled in the form of constraints in
probability or in expectation --- the property of time consistency
depends on the notion of state that one chooses, which must be suited
to the problem studied.
In particular, we have shown that, even if the ``right'' notion of state
for the class of multistage stochastic optimization problems
with a final expectation state constraint was of infinite dimension
(the conditional probability distribution of the state), it is possible to display
a state of finite dimension, so that solving
the problem by dynamic programming becomes conceivable again.

\begin{acknowledgements}
  This paper builds upon results obtained by Pierre Girardeau
  during his PhD thesis \cite{Girardeau_These_2010}, supervised
  by the three authors.
\end{acknowledgements}

\appendix
\section{An extension of Everett's theorem}
\label{ann:everett}

A result due to Everett (see~\cite{Everett_OR_1963}) links the solution of an optimization problem
under constraint and the one of the related dualized optimization
problem. We give here a slight extension, which relaxes an assumption of Everett's theorem.

Let~$\espace{U}$ be a set and let~$U\ad$ be a subset of~$\espace{U}$.
Let \( \Theta : \espace{U} \to \mathbb{R}^{m} \) be a (mulivalued) function.
We deal with the following optimization problem
\begin{equation}
  \label{pb:constrainedproblem}
  J\opt=\min_{u \in U\ad} J(u) \quad \text{s.t.} \quad
  \Theta(u) - b \leq 0 \in \mathbb{R}^{m} \eqfinv
\end{equation}
hence subject to a finite number~$m$ of inequality constraints.

\begin{theorem}
  \label{th:everett}
  Let~$\lambda \in \mathbb{R}^{m}_{+}$ be given.
  We  consider $\overline{u}^{\lambda}$, solution of the optimization problem
\begin{equation}
  \label{pb:relaxedproblem}
  \min_{u \in U\ad} J(u) + \bscal{\lambda}{\Theta(u)} \eqfinv
\end{equation}
and we set $\overline{b}^{\lambda}=\Theta(\overline{u}^{\lambda})$.
We introduce the set~$B^{\lambda} \subset \mathbb{R}^{m} $ defined by
\begin{equation}
  \label{def:Xi-lambda}
  B^{\lambda} = \defset{ b\in\mathbb{R}^{m} }{ 
    b_{i} = \overline{b}^{\lambda}_{i} \text{ if } \lambda_{i} > 0
    \;\text{ and }\;
    b_{i} \geq \overline{b}^{\lambda}_{i} \text{ if } \lambda_{i} = 0
    \eqsepv \forall i=1,\dots,m }
  \eqfinp
\end{equation}

Then,   a solution~$\overline{u}^{\lambda}$ of Problem~\eqref{pb:relaxedproblem}
  is a solution of Problem~\eqref{pb:constrainedproblem}
  for any~$b \in B^{\lambda}$.
\end{theorem}

\begin{proof}
  Let~$\overline{u}^{\lambda}$ be a solution of Problem~\eqref{pb:relaxedproblem}
  and let~$b \in B^{\lambda}$.
  We have
  \begin{align*}
    J(\overline{u}^{\lambda})
    & = \min_{u\in U\ad} J(u) + \bscal{\lambda}{\Theta(u)-\Theta(\overline{u}^{\lambda})}
      \tag{by definition of $\overline{u}^{\lambda}$} \eqfinv \\
    & = \min_{u\in U\ad} J(u) + \bscal{\lambda}{\Theta(u)-\overline{b}^{\lambda}}
      \tag{by definition of $\overline{b}^{\lambda}=\Theta(\overline{u}^{\lambda})$} \eqfinv \\
    & = \min_{u\in U\ad} J(u) + \bscal{\lambda}{\Theta(u)-b}
      \tag{by definition of~$B^{\lambda}$ in~\eqref{def:Xi-lambda}} \eqfinv \\
    & \leq \sup_{\mu \geq 0} \min_{u\in U\ad}
      J(u) + \bscal{\mu}{\Theta(u)-b}
      \tag{as $\lambda \geq 0$}
      \eqfinv \\
    & \leq \inf_{u\in U\ad} \sup_{\mu \geq 0}
      J(u) + \bscal{\mu}{\Theta(u)-b}
      \tag{by weak duality} \eqfinv \\
    & = J\opt \tag{by~\eqref{pb:constrainedproblem}}
      \eqfinp
  \end{align*}
  Since~$\overline{u}^{\lambda}$ is
  such that~$\Theta(\overline{u}^{\lambda}) = \overline{b}^{\lambda} \leq b$
  by definition of the set~$B^{\lambda}$ in~\eqref{def:Xi-lambda}, we deduce that~$\overline{u}^{\lambda}$
  is admissible for Problem~\eqref{pb:constrainedproblem}, and hence
  is an optimal solution of this problem.
\end{proof}

\section{Dynamic programming for the optimization problem
  involving martingale-type constraints}
\label{ann:progdyn-ext}

We prove in~\S\ref{sec:dp} that  Problem~\eqref{pb:constraint-equiv}
can be solved by dynamic programming under the additional assumption
that, for any time~$t$ in $\ic{0,T}$, the random variable~$\va{w}_{t}$
can take only a finite number of values. The proof is based
on a so-called interchange (between minimization and integration) lemma given
in \S\ref{sec:interchange}.

\subsection{An interchange Lemma}
\label{sec:interchange}

\begin{lemma}
  \label{le:interchange}
Let~$\espace{Y}$, $\espace{U}$, $\espace{V}$, $\espace{W}'$
and~$\espace{W}''$ be measurable spaces, and let
$\varphi : \espace{Y}\times\espace{U}\times\espace{V}\times\espace{W}''
\to \RR_+ \cup\na{+\infty}$ be a measurable extended real function.
Let~$(\omeg, \trib, \prbt)$ be a probability space\footnote{%
All random variables are defined
on~$(\omeg, \trib, \prbt)$, and we denote them using bold letters.}.
Given three random variables~$\va{y}$, $\va{w}'$ and $\va{w}''$ taking
values in~$\espace{Y}$, $\espace{W}'$ and~$\espace{W}''$ respectively,
we consider the optimization problem~${\cal P}$ defined by
\begin{subequations}
  \label{eq:pb-P}
  \begin{align}
    \Pbvalue{P}{Y}=
    \inf_{(\va{u},\va{v})} \;
    & \besp{\varphi(\va{y},\va{u},\va{v},\va{w}'')} \\
    & \text{s.t.} \quad \sigma\np{\va{u}} \subset \sigma\np{\va{w}'} \eqsepv
                        \sigma\np{\va{v}} \subset \sigma\np{\va{w}',\va{w}''}
      \label{eq:pb-P-mes} \eqfinv \\
    & \hphantom{\text{s.t.} \quad} \va{v} \text{ is integrable and }
      \bespc{\va{v}}{\sigma\np{\va{w}'}}= 0 \eqfinv
  \end{align}
\end{subequations}
where the minimization is done over couples of random variables
\( \va{u} : \omeg \to \espace{U}\) and
\( \va{v} : \omeg \to \espace{V}\).


  We define the function $\psi: \espace{Y}\to \RR_+ \cup\na{+\infty} $ by
  \begin{align}
    \psi : \espace{Y}\ni y \mapsto
    &
      \inf_{{(u,\va{v})}}
      \besp{\varphi(y,u,\va{v},\va{w}'')}
    \label{eq:Psidef}\\
    &\quad\text{s.t.}\quad \sigma\np{\va{v}} \subset \sigma\np{\va{w}''}
      \eqfinv
    \nonumber \\
    & \hphantom{\quad\text{s.t.}\quad} \va{v} \text{ is integrable and } \besp{\va{v}}= 0
      \nonumber  \eqfinp
  \end{align}
  where the minimization is done over variables \( u \in \espace{U}\)
and random variables \( \va{v} : \omeg \to \espace{V}\).

  We suppose that the two random
  variables $\va{w}'$ and $\va{w}''$ are independent, that they
  each take a finite number of values, and that the
  random variable~$\va{y}$ is $\sigma\np{\va{w}'}$-measurable, that is,
  \( \sigma\np{\va{y}}  \subset \sigma\np{\va{w}'} \).
  Then, the optimal value~$\Pbvalue{P}{Y}$ of Problem~${\cal P}$
  satisfies the following interchange formula 
    \begin{equation}
    \label{eq:interchange}
      \Pbvalue{P}{Y}
      = \besp{\psi(\va{y})}
      \eqfinp
    \end{equation}

  \end{lemma}

  \begin{proof}
    Letting $\na{w'_{i}}_{i\in \ic{0,N'}}$ and $ \na{w''_{i}}_{i\in \ic{0,N''}}$
be the sets of values taken by, respectively, the random variables $\va{w}'$
and $\va{w}''$, 
we denote
\begin{align}
  \va{w}': \Omega \to \na{w'_{i}}_{i\in \ic{0,N'}}
  &\quad\text{with}\quad \bprob{\na{\va{w}= w'_i}} = \pi'_i\eqsepv \forall i  \in \ic{0,N'}
    \label{eq:Wprimediscret} \eqfinv
  \\
  \va{w}'': \Omega \to \na{w''_{i}}_{i\in \ic{0,N''}}
  &\quad\text{with}\quad \bprob{\na{\va{w}= w''_i}} = \pi''_i\eqsepv \forall i  \in \ic{0,N''}
    \label{eq:Wprimeprimediscret} \eqfinp
\end{align}
Now, since~$\va{y}$ is a~$\sigma(\va{w}')$-measurable
random variable and from the measurability constraints on the random variables
$\va{u}$ and $\va{v}$ in Equation~\eqref{eq:pb-P-mes}, we can represent
these random variables as follows:
\begin{align}
  \va{y} = \sum_{i=0}^{N'} y_i \findi{w'_i}(\va{w}')
  \eqsepv
  \va{u} = \sum_{i=0}^{N'} u_i \findi{w'_i}(\va{w}')
  \eqsepv
  \va{v} = \sum_{i=0}^{N'}\sum_{j=0}^{N''} v_{i,j} \findi{w'_i}(\va{w}')\findi{w''_j}(\va{w}'')
  \eqfinp
  \label{eq:representation}
\end{align}
We have just expressed the fact that the set
of $\sigma\np{\va{w}}$-measurable random variables taking values in
a set ${\mathbb F}$ is in bijection with the product space ${\mathbb F}^{N}$
if the random variable $\va{w}$ takes $N$ different values.

We start the proof by using Equation~\eqref{eq:representation} to establish the following equalities
\begin{subequations}
  \label{eq:discrete-equations}
  \begin{align}
    \bespc{\varphi(\va{y}, \va{u},\va{v},\va{w}'')}{\va{w}'}
    &=
      \sumi{\Bp{\pisumj{\varphi\Bp{y_i, u_i,v_{i,j},w''_j}}}}
      \eqfinv
      \label{eq:EspcondPhi}\\
    \besp{\varphi(\va{y},\va{u},\va{v},\va{w}'')}
    &=
        \pisumi{\Bp{\pisumj{\varphi\Bp{y_i, u_i,v_{i,j},w''_j}}}}
      \eqfinv
      \label{eq:EspPhi}\\
    \bespc{\va{v}}{\va{w}'}
    &=
      \sumi{\bgp{\pisumj{v_{i,j}}}}
      \label{eq:EspcondV}\eqfinp
  \end{align}
\end{subequations}
All the manipulations below are easy to check, and are justified
 because all quantities take extended nonnegative values.

$\bullet$ For Equation~\eqref{eq:EspcondPhi}:
  \begin{subequations}
    \begin{align*}
      \EE\big(
      &\varphi(
        \va{y},\va{u},\va{v},\va{w}'')\;\big\vert\; {\va{w}'} \big) \\
      & = \bgespc{\varphi\Bp{\sumi{y_i},\sumi{u_i},\sumij{v_{i,j}},\va{w}''}}{\va{w}'}
        \tag{by~\eqref{eq:representation}} \\
      &=
        \bgespc{\sumi{\varphi\Bp{y_i, u_i, \sumj{v_{i,j}},\va{w}''}}}{\va{w}'}
      \\
      &=
        \sumi{
        \bgespc{\varphi\Bp{y_i, u_i,\sumj{v_{i,j}}, \va{w}''}}{\va{w}'}
        }
      \\
      &=
        \sumi{
        \bgesp{\varphi\Bp{y_i, u_i,\sumj{v_{i,j}}, \va{w}''}}
        }
        \tag{as $\va{w}'$ and $\va{w}''$ are independent}
      \\
      &=
        \sumi{
        \bgesp{\sumj{\varphi\Bp{y_i, u_i,v_{i,j},w''_j}}}
        }
      \\
      &=
        \sumi{\Bp{\pisumj{\varphi\Bp{y_i, u_i,v_{i,j},w''_j}}}}
        \tag{by~\eqref{eq:Wprimeprimediscret}} \eqfinp
    \end{align*}
  \end{subequations}

  $\bullet$ For Equation~\eqref{eq:EspPhi}:
  \begin{subequations}
    \begin{align*}
      \besp{\varphi(\va{y},\va{u},\va{v},\va{w}'')}
      &=
        \Besp{\bespc{\varphi(\va{y},\va{u},\va{v},\va{w}'')}{\va{w}'}}
      \\
      &=
        \Besp{
        \sumi{\Bp{\pisumj{\varphi\Bp{y_i, u_i,v_{i,j},w''_j}}}}
        }
        \tag{by~\eqref{eq:EspcondPhi}} \\
      &=
        \pisumi{\Bp{\pisumj{\varphi\Bp{y_i, u_i,v_{i,j},w''_j}}}}
        \tag{by~\eqref{eq:Wprimediscret}} \eqfinp
    \end{align*}
  \end{subequations}

  $\bullet$ For Equation~\eqref{eq:EspcondV}:
  \begin{subequations}
    \begin{align*}
      \bespc{\va{v}}{\va{w}'}
      &=
        \bgespc{\sumij{v_{i,j}}}{\va{w}'}
        \tag{by~\eqref{eq:representation}} \\
      &=
        \sumi{\bgespc{\sumj{v_{i,j}}}{\va{w}'}}
      \\
      &=
        \sumi{\bgesp{\sumj{v_{i,j}}}}
        \tag{as $\va{w}'$ and $\va{w}''$ are independent}
      \\
      &=
        \sumi{\bgp{\pisumj{v_{i,j}}}}
        \tag{by~\eqref{eq:Wprimeprimediscret}} \eqfinp
    \end{align*}
  \end{subequations}

  Using Equations~\eqref{eq:representation} and~\eqref{eq:EspcondV},
  for any $\sigma\np{\va{W}''}$-measurable random variable $\va{V}$,
  we have the equivalence
  \begin{equation}
  \label{eq:contr-espcondv}
    \bespc{\va{v}}{\va{w}'} = 0
    \iff
    \pisumj{v_{i,j}} = 0\eqsepv \; \forall i \in \ic{0,N'}
    \eqfinp
  \end{equation}
  Using again Equations~\eqref{eq:representation}
  and~\eqref{eq:discrete-equations}, we obtain that
  the optimization Problem~\eqref{eq:pb-P} is equivalent
  to the following optimization problem
  \begin{subequations}
    \label{pb:discrete}
    \begin{align}
    \inf_{\substack{\na{u_i}_{i\in \ic{0,N'}} \\
                        \na{v_{i,j}}_{i\in \ic{0,N'}, j \in \ic{0,N''}}}}
    & \sum_{i=0}^{N'}
      \pi'_i
      \Bp{\sum_{j=0}^{N''} \pi''_j \varphi\np{y_i, u_i, v_{i,j}, w''_i}}
    \\
    & \text{s.t.} \quad
      {\sum_{j=0}^{N''} \pi''_j v_{i,j} = 0\eqsepv \; \forall i \in \ic{0,N'}}
      \eqfinp
    \end{align}
  \end{subequations}
The optimization problem~\eqref{pb:discrete} trivially splits into
a family~$\na{{\cal P}_i}_{i\in\ic{0,N'}}$ of $N'$ independent
optimization problems, Problem~${\cal P}_i$ being defined by
\begin{subequations}
  \begin{align}
    \Pbvalueter{{\cal P}}{y_i}=
    & \inf_{(u_i,\na{v_{i,j}}_{j \in \ic{0,N''}})}
      {\sum_{j=0}^{N''} \pi''_j \varphi\np{y_i, u_i, v_{i,j}, w''_i}}
 \\
    & \text{s.t.}\quad
      \sum_{j=0}^{N''} \pi''_j v_{i,j} = 0 \eqsepv
  \end{align}
  \label{eq:Psidef_finite}
\end{subequations}
and the value of Problem~\eqref{pb:discrete} is the weighted sum
of the values of the family of problems $\na{{\cal P}_i}_{i \in \ic{0,N'}}$:
\begin{equation*}
\Pbvalue{P}{Y}=\sum_{i=0}^{N'} \pi'_i \Pbvalueter{{\cal P}}{y_i} \eqfinp
\end{equation*}
We notice that \( \Pbvalueter{{\cal P}}{y_i} \) in~\eqref{eq:Psidef_finite} is
exactly \( \psi\np{y_i} \) in~\eqref{eq:Psidef}, so that the above equation
gives~\eqref{eq:interchange}.
\end{proof}

\subsection{Proof of Theorem~\ref{th:constraint-equivBellman}}
\label{sec:dp}

\begin{proof}
  For any $\tau \in \ic{0,T}$ we consider the minimization Problem~${\cal P}_{\tau}$
  defined by\footnote{%
    For $\tau=0$, the value of Problem~${\cal P}_{0}$ is simply
    \( \mathscr{V}_{0}({x}_{0},{z}_{0}) \).}
  \begin{subequations}
    \label{pb:constraint-equiv-tau}
    \begin{align}
      \min_{(\va{u},\va{v},\va{x},\va{z})}
      &
        \bgesp{\sum_{t=0}^{\tau{-}1}
        L_{t}(\va{x}_{t},\va{u}_{t},\va{w}_{t+1})+
        \mathscr{V}_{\tau}(\va{x}_{\tau},\va{z}_{\tau})}
        \eqfinv \label{pb:constraint-equiv-tau-cost}
      \\
      & \text{s.t.}\; \va{x}_{0} = x_{0} \eqsepv \va{z}_{0} = z_{0} \eqsepv
        \label{pb:constraint-equiv-tau-init}    \\
      &\text{and, for all $t \in\ic{0,\tau-1}$,}
        \nonumber \\
      &\hphantom{\text{and}\;\forall t}
        \va{x}_{t+1} = f_{t}(\va{x}_{t},\va{u}_{t},\va{w}_{t+1})
        \eqsepv
        \va{z}_{t+1} = \va{z}_{t} + \va{v}_{t}
        \eqfinv \label{pb:constraint-equiv-tau-dyn}
      \\
      &\hphantom{\text{and}\;\forall t}
        \sigma\bp{\va{u}_{t}} \subset
        \sigma\bp{
        \va{w}_{1},\dots,\va{w}_t}
        \eqsepv
        \sigma\bp{\va{v}_{t}} \subset
        \sigma\bp{
        \va{w}_{1},\dots,\va{w}_{t+1}}
        \eqfinv \label{pb:constraint-equiv-tau-mes}
      \\
      &\hphantom{\text{and}\;\forall t}
        \text{$\va{v}_{t}$ is integrable and } \;
        \bespc{\va{v}_{t}}
        {{
        \va{w}_{1},\dots,\va{w}_t}} = 0
        \eqfinv  \label{pb:constraint-equiv-tau-mart}
    \end{align}
  where the sequence $\ba{\mathscr{V}_\tau}_{\tau\in\ic{0,T}}$ of value
  functions, with \( \mathscr{V}_\tau : \espacea{X}_\tau \times\espacea{Z}_\tau \to \RR_+ \cup\na{+\infty} \),
  appearing in the cost function~\eqref{pb:constraint-equiv-tau-cost}
  is given by the Bellman recursion~\eqref{eq:constraint-equivBellman}.
  To simplify the notation, we denote by $\Lambda_{\tau}$ the set
  of random variables $\np{\va{u}_t,\va{v}_t}_{t\in \ic{0,\tau}}$ and
  $\np{\va{x}_t,\va{z}_t}_{t\in \ic{0,\tau+1}}$ satisfying the constraints
  \eqref{pb:constraint-equiv-tau-init} -- \eqref{pb:constraint-equiv-tau-dyn} --
  \eqref{pb:constraint-equiv-tau-mes} -- \eqref{pb:constraint-equiv-tau-mart}.
  We recall that, by Equation~\eqref{eq:sigma-field-t}, $\tribu{F}_{\tau}$
  represents the $\sigma$-field generated by $\np{\va{w}_1,\ldots,\va{w}_{\tau}}$
  for all $\tau \in \ic{1,T}$.
  \end{subequations}

  We are now going to prove, by backward induction,
  that the value of Problem~\eqref{pb:constraint-equiv} with~$t_{0}=0$ is equal
  to the value of Problem ${\cal P}_{\tau}$ in~\eqref{pb:constraint-equiv-tau} for any $\tau \in \ic{0,T}$.

  First, the value of Problem~\eqref{pb:constraint-equiv}, with~$t_{0}=0$, is equal
  to the value of Problem ${\cal P}_{\tau}$ in~\eqref{pb:constraint-equiv-tau} for $\tau=T$.
  Indeed, the criterion~\eqref{pb:constraint-equiv-tau-cost} in
  Problem~\eqref{pb:constraint-equiv-tau}, satisfies, for~$\tau=T$,
  \begin{align*}
   &  \bgesp{\sum_{t=0}^{T{-}1}
        L_{t}(\va{x}_{t},\va{u}_{t},\va{w}_{t+1})+
     \mathscr{V}_{T}(\va{x}_{T},\va{z}_{T})}
     \\
    &=
     \bgesp{\sum_{t=0}^{T{-}1}
      L_{t}(\va{x}_{t},\va{u}_{t},\va{w}_{t+1})+
      \bp{K(x) + \fcara{\na{g(x) - z \leq 0}}}(\va{x}_{T},\va{z}_{T})}
      \tag{as~$\mathscr{V}_{T}$ is defined by~\eqref{eq:constraint-equivBellman-T}}
    \\
    &=
     \bgesp{\sum_{t=0}^{T{-}1}
      L_{t}(\va{x}_{t},\va{u}_{t},\va{w}_{t+1})+ K(\va{x}_{T}) }
      \mtext{ under the almost sure final constraint}
    \\
    &\hphantom{\bgesp{\sum_{t=0}^{T{-}1}
      L_{t}(\va{x}_{t},\va{u}_{t},\va{w}_{t+1})+ K(\va{x}_{T}) }
      \mtext{ under the}}
      g(\va{x}_{T})-\va{z}_{T} \leq 0
      \eqfinp
  \end{align*}
Thus, we obtain that
  Problem~\eqref{pb:constraint-equiv-tau} for~$\tau=T$ is
  the same as Problem~\eqref{pb:constraint-equiv} with~$t_{0}=0$,
  the only difference being that
  the almost sure final constraint~\eqref{pb:constraint-equiv-finl}
  has been moved in the final cost in~\eqref{pb:constraint-equiv-tau-cost}.

Second, we prove by backward induction that
the value of Problem~\eqref{pb:constraint-equiv} is equal
  to the value of Problem ${\cal P}_{\tau}$ for any $\tau \in \ic{0,T}$.
  For this purpose, assuming that the value of Problem~\eqref{pb:constraint-equiv} is equal
  to the value of Problem ${\cal P}_{\tau+1}$, we prove that it is also
  equal to the value of Problem ${\cal P}_{\tau}$.
  We immediately get that
    \begin{align}
      &
        \min_{\bp{\np{\va{u}_t,\va{v}_t}_{t\in \ic{0,\tau}},
                  \np{\va{x}_t,\va{z}_t}_{t\in \ic{0,\tau+1}}} \in \Lambda_{\tau}}
        \bgesp{\sum_{t=0}^{\tau}
        L_{t}(\va{x}_{t},\va{u}_{t},\va{w}_{t+1})+
        \mathscr{V}_{\tau+1}(\va{x}_{\tau+1},\va{z}_{\tau+1})}
        \nonumber
      \\
      &\quad=
        \min_{\bp{\np{\va{u}_t,\va{v}_t}_{t\in \ic{0,\tau-1}},
                  \np{\va{x}_t,\va{z}_t}_{t\in \ic{0,\tau}}} \in \Lambda_{\tau-1}}
        \Bigg(
        \bgesp{\sum_{t=0}^{\tau-1}
        L_{t}(\va{x}_{t},\va{u}_{t},\va{w}_{t+1})}
        \label{eq:step-1} 
      \\
      &\hphantom{\quad=\quad}
        + \min_{
        \substack{ \np{\va{u}_{\tau},\va{v}_{\tau}} \\
      \sigma\np{\va{u}_{\tau}} \subset \tribu{F}_{\tau} \\ 
      \sigma\np{\va{v}_{\tau}} \subset \tribu{F}_{\tau+1} \\ 
      \nespc{\va{v}_{\tau}}{\tribu{F}_{\tau}}=0}} 
      \bgesp{
         L_{\tau}(\va{x}_{\tau},\va{u}_{\tau},\va{w}_{\tau+1})+
        \mathscr{V}_{\tau+1}\bp{f_{\tau}(\va{x}_{\tau},\va{u}_{\tau},\va{w}_{\tau+1}),\va{z}_{\tau}+\va{v}_{\tau}}}
      \Bigg)
      \eqfinv
      \nonumber
    \end{align}
    because all quantities take extended nonnegative values.

    Now, we apply Lemma~\ref{le:interchange}
    to the inner minimization, with $\va{y}=(\va{x}_{\tau},\va{z}_{\tau})$,
    $\va{w}'=(\va{w}_1,\ldots,\va{w}_{\tau})$, $\va{w}''=\va{w}_{\tau+1}$ and
    with the function $\varphi\bp{(x,z),u,v,w''}=L_{\tau}(x,u,w'')$ $+
    \mathscr{V}_{\tau+1}\big(f_{\tau}(x,u,w'')$, $z+ v\big)$ and deduce that
    \begin{align}
      \min_{
      \substack{ \np{\va{u}_{\tau},\va{v}_{\tau}} \\
      \sigma\np{\va{u}_{\tau}} \subset \tribu{F}_{\tau} \\
      \sigma\np{\va{v}_{\tau}} \subset \tribu{F}_{\tau+1} \\
      \nespc{\va{v}_{\tau}}{\tribu{F}_{\tau}}=0}}
      \Besp{
      L_{\tau}(\va{x}_{\tau},\va{u}_{\tau},\va{w}_{\tau+1})+
      \mathscr{V}_{\tau+1}\bp{f_{\tau}(\va{x}_{\tau},\va{u}_{\tau},\va{w}_{\tau+1}),
                              \va{z}_{\tau}+\va{v}_{\tau}}}
      =
      \besp{ \mathscr{V}_{\tau}\np{\va{x}_{\tau},\va{z}_{\tau}}}
      \eqfinv \nonumber
    \end{align}
    because~$\psi(y)$ in~\eqref{eq:Psidef} is exactly
    $\mathscr{V}_{\tau}\np{x,z}$ in~\eqref{eq:constraint-equivBellman-t}.
  Combined with Equation~\eqref{eq:step-1}, this leads to
  \begin{align}
    &
      \min_{
      \bp{\np{\va{u}_t,\va{v}_t}_{t\in \ic{0,\tau}},
                  \np{\va{x}_t,\va{z}_t}_{t\in \ic{0,\tau+1}}}\in \Lambda_{\tau}}
      \bgesp{\sum_{t=0}^{\tau}
      L_{t}(\va{x}_{t},\va{u}_{t},\va{w}_{t+1})+
      \mathscr{V}_{\tau+1}(\va{x}_{\tau+1},\va{z}_{\tau+1})}
      \nonumber
    \\
    &\hphantom{\quad}
     = \min_{\bp{\np{\va{u}_t,\va{v}_t}_{t\in \ic{0,\tau-1}},
                  \np{\va{x}_t,\va{z}_t}_{t\in \ic{0,\tau}}} \in \Lambda_{\tau-1}}
      \bgesp{\sum_{t=0}^{\tau-1}
      L_{t}(\va{x}_{t},\va{u}_{t},\va{w}_{t+1}) +
      \mathscr{V}_{\tau}\bp{\va{x}_{\tau},\va{z}_{\tau}}}
      \eqfinp
      \nonumber
  \end{align}
  We conclude that the value of Problem~\eqref{pb:constraint-equiv}
  is equal to the value of Problem ${\cal P}_{\tau}$, so that we have
  by induction that the value of Problem~\eqref{pb:constraint-equiv}
  is equal to the value of Problem ${\cal P}_{\tau}$
  in~\eqref{pb:constraint-equiv-tau} for any $\tau \in \ic{0,T}$.

  The value of Problem~\eqref{pb:constraint-equiv} is thus equal to
  the value of problem ${\cal P}_{0}$, namely $\mathscr{V}_{0}(x_0,z_0)$,
  and can therefore be computed by using the dynamic programming
  equation~\eqref{eq:constraint-equivBellman}.
\end{proof}

\newcommand{\noopsort}[1]{} \ifx\undefined\allcaps\def\allcaps#1{#1}\fi


\begin{thebibliography}{10}

\bibitem{Artzner_AOR_2007}
P.~Artzner, F.~Delbaen, J.-M. Eber, D.~Heath, and H.~Ku.
\newblock Coherent multiperiod risk-adjusted values and {B}ellman's principle.
\newblock {\em Annals of Operations Research}, 152(1):5--22, July 2007.

\bibitem{Bertsekas_AS_2005}
D.~P. Bertsekas.
\newblock {\em Dynamic Programming and Optimal Control, Vol. I}.
\newblock Athena Scientific, Belmont, Massachusets, second edition, 2005.

\bibitem{Bertsekas-Shreve:1996}
D.~P. Bertsekas and S.~E. Shreve.
\newblock {\em Stochastic Optimal Control: The Discrete-Time Case}.
\newblock Athena Scientific, Belmont, Massachusetts, 1996.

\bibitem{Bouchard_SIAM_2009}
B.~Bouchard, R.~Elie, and N.~Touzi.
\newblock Stochastic target problems with controlled loss.
\newblock {\em SIAM Journal on Control and Optimization}, 48(5):3123--3150,
  2009.

\bibitem{Carpentier-Chancelier-Cohen-DeLara-Girardeau:2012}
P.~Carpentier, J.-P. Chancelier, G.~Cohen, M.~{De Lara}, and P.~Girardeau.
\newblock Dynamic consistency for stochastic optimal control problems.
\newblock {\em Annals of Operation Research}, 200(1):247--263, 2012.

\bibitem{Cheridito_EJP_2006}
P.~Cheridito, F.~Delbaen, and M.~Kupper.
\newblock Dynamic monetary risk measures for bounded discrete-time processes.
\newblock {\em Electronic Journal of Probability}, 11(3):57--106, 2006.

\bibitem{Detlefsen_FS_2005}
K.~Detlefsen and G.~Scandolo.
\newblock Conditional and dynamic convex risk measures.
\newblock {\em Finance and Stochastics}, 9(4):539--561, October 2005.

\bibitem{Everett_OR_1963}
H.~Everett.
\newblock Generalized {L}agrange multiplier method for solving problems of
  optimum allocation of resources.
\newblock {\em Operations research}, 11:399--417, 1963.

\bibitem{Girardeau_These_2010}
P.~Girardeau.
\newblock {\em R\'{e}solution de grands probl\`{e}mes en optimisation
  stochastique dynamique et synth\`{e}se de lois de commande}.
\newblock PhD thesis, Universit\'{e} Paris-Est, december 2010.
\newblock \url{http://www.theses.fr/2010PEST1026}.

\bibitem{Granato_These_2012}
G.~Granato.
\newblock {\em Optimisation de lois de gestion \'{e}nerg\'{e}tiques des
  v\'{e}hicules hybrides}.
\newblock PhD thesis, \'{E}cole Polytechnique, december 2012.

\bibitem{Hammond:1989}
P.~J. Hammond.
\newblock Consistent plans, consequentialism, and expected utility.
\newblock {\em Econometrica}, 57(6):1445--1449, 1989.

\bibitem{Pfeiffer_AMO_2018}
L.~Pfeiffer.
\newblock Two approaches to stochastic optimal control problems with a
  final-time expectation constraint.
\newblock {\em Appl Math Optim}, 77:377--404, 2018.

\bibitem{Pflug2016}
G.~C. Pflug and A.~Pichler.
\newblock Time-inconsistent multistage stochastic programs: Martingale bounds.
\newblock {\em European Journal of Operational Research}, 249(1):155--163,
  2016.

\bibitem{Riedel_SPA_2004}
F.~Riedel.
\newblock Dynamic coherent risk measures.
\newblock {\em Stochastic Processes and their Applications}, 112(2):185 -- 200,
  2004.

\bibitem{Ruszczynski_OO_2009}
A.~Ruszczynski.
\newblock Risk-averse dynamic programming for {M}arkov decision processes.
\newblock {\em Mathematical Programming}, 125:235--261, 2010.

\bibitem{Shapiro_ORL_2009}
A.~Shapiro.
\newblock On a time consistency concept in risk averse multistage stochastic
  programming.
\newblock {\em Operations Research Letters}, 37(3):143 -- 147, 2009.

\end{thebibliography}
\end{document}